\documentclass[reqno]{amsart}
\usepackage{amssymb}
\usepackage{enumitem}
\usepackage{color}

\pdfpkmode={supre}
\usepackage{hyperref}
\hypersetup{
    colorlinks,
    citecolor=red,
    filecolor=black,
    linkcolor=blue,
    urlcolor=blue
}

\newtheorem{theorem}{Theorem}[section]
\newtheorem{lemma}[theorem]{Lemma}
\newtheorem{proposition}[theorem]{Proposition}
\newtheorem{corollary}[theorem]{Corollary}
\newtheorem{definition}[theorem]{Definition}

\newtheorem{remark}[theorem]{Remark}

\numberwithin{equation}{section}
\numberwithin{theorem}{section}
\newcommand{\mc}[1]{{\mathcal #1}}

\newcommand{\bb}[1]{{\mathbb #1}}

\newcommand{\eps}{\varepsilon}
\newcommand{\un}[1]{{\mathbf 1}_{#1}}
\newcommand{\Glimsup}{\mathop{\textrm{$\Gamma\!\!$--$\varlimsup$}}}
\newcommand{\Gliminf}{\mathop{\textrm{$\Gamma\!\!$--$\varliminf$}}}
\newcommand{\Glim}{\mathop{\textrm{$\Gamma\!\!$--$\lim$}}}
\newcommand{\bel}[2]{\begin{equation} \label{#1} \begin{split} #2
 \end{split} \end{equation}}
\newcommand{\dangle}[2]{\langle #1,\,#2 \rangle}

\DeclareMathOperator{\hess}{Hess}
\newcommand{\undn}{\underline{n}}

\author[G.Di Ges\`u]{Giacomo Di Ges\`u}
\address{ Giacomo Di Ges\`u, TU Vienna, 
E101 - Institut f\"ur Analysis und Scientific Computing, 
Wiedner Hauptstr.\ 8, 
1040 Vienna, Austria.}
\email{giacomo.di.gesu@tuwien.ac.at}

\author[M.Mariani]{Mauro Mariani}
\address{Mauro Mariani, Dipartimento di Matematica, Universit\`a degli Studi di Roma La Sapienza, Piazzale Aldo Moro 5, 00185, Roma, Italy.}
\email{mariani@mat.uniroma1.t}

\title{Full metastable asymptotic of the Fisher information}

\begin{document}

\keywords{Fisher Information, Metastability, $\Gamma$-convergence, Semiclassical spectral theory, Witten Laplacian.}
\subjclass[2010]{35P20,28D20,49J45,60F10}
\begin{abstract}
We establish an expansion by $\Gamma$-convergence of the Fisher information relative to
the reference measure $e^{-\beta V}dx$, where
$V$ is a generic multiwell potential and $\beta\to\infty$. 
The expansion reveals a  hierarchy of scales
reflecting the metastable behavior of the underlying overdamped Langevin dynamics: 
distinct scales emerge and become relevant  
depending on whether one considers probability measures 
concentrated on local minima of $V$, probability measures concentrated on critical points of $V$, or generic probability measures on $\bb R^d$. 
We thus fully describe the asymptotic behavior of minima 
of the Fisher information over regular sets of probabilities.
 The analysis mostly relies on spectral properties of diffusion operators and the 
related semiclassical Witten Laplacian  and also covers the case of a compact smooth manifold as underlying space.
\end{abstract}

\maketitle

\section{Introduction}
\label{s:intro}
The Fisher information of a probability measure $\mu $ relative to a reference measure $m$ on a smooth manifold is given by the expression
\bel{e:relFisher}{ \mc I(\mu|m)  := 
\tfrac{1}{2} \int \frac{|\nabla \varrho|^2}{\varrho} \, dm      =     2 \int     |\nabla \sqrt{\varrho}|^2 \, dm     ,   
}
assuming that $\varrho: = \frac{d\mu}{dm} $ exists and is sufficiently regular (see the precise
Definition~\ref{d:dvmore} below). This is a classical and ubiquitous object measuring the discrepancy between two measures and appearing in various fields as Statistics,
Information Theory and Statistical Mechanics.

In the context of Statistical mechanics the reference measure usually appears in the form $m_\beta = e^{-\beta V}dx$,
where $V$ is a suitable potential describing the interaction in the model, $\beta$ is a positive parameter, proportional to the inverse temperature
and $dx$ is the volume measure corresponding to the absence of interaction.   The main interest often lies in the dynamical properties of the associated heat flow evolution featuring $m_\beta$ as stationary measure. Under suitable regularity assumptions on $V$ it may 
be described by the Fokker-Planck equation 
\bel{e:fokkerplanckev}{
\partial_t u =   \beta^{-1}\Delta u +   \mathrm{div}(u\nabla V ) 
}
or, on a pathwise level, by the overdamped Langevin equation 
\bel{e:diffusionprocess}{
\dot X_t =- \nabla V (X_t)+\sqrt{2\beta^{-1}} \, \dot W_t    ,
}
where $W_t$ is Brownian motion.
The Fisher information may enter the description and analysis of a Statistical Mechanics model in various ways: through functional inequalities, prominently the Log-Sobolev and HWI inequality (see e.g. \cite{BGL} and references therein and in particular~\cite{MS} for the $\beta\to\infty$ regime considered here); as rate functional in Large deviation principles describing fluctuations of the empirical occupation measure
of~\eqref{e:diffusionprocess} when $t\to\infty$, \cite{DV};
or also as a tool to give an alternative construction of the dynamics~\eqref{e:fokkerplanckev} as gradient flow with respect to the relative entropy functional 
\cite{JKO} (see also the discussion at the end of this introduction). 
\subsection{The low temperature regime and metastability}
\label{ss:lowt}
In this paper we study the asymptotic behaviour of the Fisher information relative to $m_\beta= e^{-\beta V}dx$ in the low temperature regime, that is $\beta\to \infty$. When the potential $V$ admits several local minima, in this regime energy traps are created which slow down the dynamics~\eqref{e:diffusionprocess} around these minima and produce metastability effects.

In the heuristic picture,  there are several relevant time scales that feature a non-trivial dynamics: 
on the natural time-scale on which the process is defined, the system essentially follows the
deterministic gradient flow with respect to the potential $V$. 
In particular, all critical points of V are absorbing for the limiting dynamics.
On a longer time scale (see below), the process just 
lives on critical points; however on this scale stable critical points (that is, local minima of $V$) are absorbing for the dynamics (once visited, the system sticks there).
When observed on even larger  time scales, typically exponentially large in $\beta$ \cite[Chap. 6.]{FW}, even among the local minima distinctions become observable: 
on the long run, rare but sufficiently large fluctuations may occur that allow the process to climb the mountain pass which leads to another local minimum. 
This tunneling or metastable transition among local minima favors the passage from a minimum point to a deeper, energetically more convenient minimum point. 
Thus, on the long run, the deeper the local minimum, the more time the process will spend around it.

The reversible diffusion model \eqref{e:diffusionprocess} is a paradigmatic model for metastable phenomena, that show up in the dynamical behavior of a large variety of complex real world systems, and can be regarded as a mark of dynamical phase transitions. It is often a detailed theoretical and numerical analysis of metastability that allows for a proper upscaling of microscopic models - defined at the atomic or molecular level -  to macroscopic descriptions catching experimental observations in materials science, biology and chemistry.
 We refer to Kramers' influential work~\cite{Kr}, to~\cite{HTB} and to~\cite{Be,Le, DLPN} for recent brief reviews on some of the mathematical techniques used in the analysis of metastability. We also refer specifically to \cite{Mat, OV} and references therein for some classical rigorous results concerning the heuristic picture described above.

 \
 
One expects the metastable behavior to be encoded somehow in the asymptotic behavior of the Fisher information. We discuss below why $\Gamma$-convergence is the natural tool to characterize this asymptotic behavior, in particular in the context of metastability and convergence of multi-scale reversible dynamics. However, for the sake of readability, we first shortly describe our main result in Section~\ref{ss:intromain} below, then informally explain how $\Gamma$-convergence can be used as a solid theoretical framework for metastability in Section~\ref{ss:kifer}, in particular in the context of large deviations of the occupation measure. We quickly review the applications of $\Gamma$-convergence in multi-scale gradient flows in Section~\ref{ss:multigrad} (see also \cite{PSV} where  $\Gamma$-convergence is exploited to upscale an essentially one-dimensional metastable Fokker-Planck equation). We recall the definition of $\Gamma$-convergence in Section~\ref{ss:gammaconv} below, and refer to \cite{Br} for a systematic treatment.

\subsection{Informal description of the main result}
\label{ss:intromain}
In this paper a full expansion in the sense of $\Gamma$-convergence of the Fisher information is established, which indeed features the complete cascade of metastable scales as informally described above. Using standard techniques, it is not hard to see that, under mild assumptions on the potential $V$, the rescaled Fisher information 
\bel{e:rescaledFisher}{
I_\beta:=        \tfrac{1}{\beta^{2}}     \mc I(\cdot | e^{-\beta V} dx)  
}
$\Gamma$-converges (in the weak topology of probabilities) as $\beta\to \infty$ to the linear functional
\[I(\mu):=\tfrac 12\int d\mu(x) |\nabla V|^2(x).\]    
On the dynamical level this corresponds to the convergence of the diffusion \eqref{e:fokkerplanckev}-\eqref{e:diffusionprocess} to the deterministic transport along the gradient of $V$. 
Note that $Z_\beta^{-1}e^{-\beta V} dx$ is the unique minimizer of $I_\beta$, provided  $Z_\beta=\exp(-\beta V)<\infty$. 
On the other hand, if $V$ admits several critical points, $I(\mu)$ vanishes on any probability that is concentrated on such points.

This suggests that one may obtain a non-trivial $\Gamma$-limit also when multiplying $I_\beta$ by a suitable diverging constant, so that, according to the heuristic
picture described above, the new limiting functional vanishes only on probability measures concentrated on local minima. This is indeed the case, and $\beta I_\beta$ also $\Gamma$-converges to a functional $J$ with the mentioned properties
(see Theorem~\ref{t:main} for its explicit characterization). 

To capture the exponential scales associated with tunneling between local minima we 
 iterate the same procedure: assuming that there is a unique global minimum point $x_0$ of $V$
 and labelling by $\{ x_1, \dots, x_n\}$ the other local minimum points of $V$, we
 investigate for each $k=1, \dots, n$ the $\Gamma$-limit of 
$\beta e^{\beta W_k} I_\beta$, where $W_k$ is the mountain pass barrier which separates $x_k$ from 
a deeper minimum (see Section~\ref{s:main} for precise defintions).
 Under genericity assumptions, we prove that  for each $k=1, \dots, n$ also $\beta e^{\beta W_k}I_\beta$ $\Gamma$-converges to some $J_k$, which
 can again be explicitly characterized  (see Theorem~\ref{t:main}) and involves the prefactors already appearing in the famous Kramers formula
 for metastable critical times \cite{Kr} (see also the $\Gamma-$convergence result in \cite{PSV}). Since 
 $J_k(\mu)=0$ for all $k$'s iff $\mu=\delta_{x_0}$, this amounts to a full expansion of $I_\beta$ by $\Gamma$-convergence
 (see \cite[Chap.~1.10]{Br} and Section~\ref{s:main} of this paper). In other terms, as far as infima over closed and open sets are concerned (that is, in the sense of $\Gamma$-convergence)
\bel{e:sviluppo1}{
I_\beta(\mu)\sim I(\mu)+ \tfrac{1}{\beta} J(\mu)+\sum_{k=1}^n \tfrac{1}\beta\,e^{-\beta\,W_k}J_k(\mu)   .  
}

As a consequence of~\eqref{e:sviluppo1} and the Donsker-Varadhan Large Deviations Principle \cite{DV} 
we infer that, in the sense of large deviations, the solution to \eqref{e:diffusionprocess} satisfies  (see Corollary~\ref{c:ld})
\bel{e:ld1bis}{
\bb P\Big(\tfrac{1}{T}\! \int_0^T\!\!\!\!\!ds f(X_s) \sim \int \! d\mu(x) f(x), \, \forall f\Big)
\sim 
e^{-\beta T\left(I(\mu)+\tfrac{1}{\beta} J(\mu)+\sum_{k=1}^n \frac{e^{-\beta\,W_k}}{\beta} J_k(\mu) \right)}
.
}
That is, we characterize the sharp asymptotic of the left hand side for every $\mu\in \mc P(\bb R^d)$.

To prove our main result~\eqref{e:sviluppo1} we take a spectral point of view and exploit results which were mainly developed in the
context of the semiclassical spectral analysis of
Schr\"{o}dinger operators. Indeed, as highlighted by the expression on the right hand side of~\eqref{e:relFisher}, the Fisher information can be seen as a Dirichlet form on $L^2(dm)$ and thus be studied from a spectral point of view by considering the corresponding selfadjoint operator.
 When $m_\beta = e^{-\beta V}dx$
the latter is given by the diffusion operator
 \bel{e:generatorL}{       L_\beta  =     \Delta     -    \beta \nabla V \cdot \nabla        .  }
It is well-known that a unitary transformation, sometimes called ground state transformation,  turns  the generator into a Schr\"{o}dinger operator
acting now on the flat space $L^2(dx)$. More precisely,     
  \bel{e:Witten}{      - e^{ -  \tfrac 12\beta V}     L_\beta e^{ \tfrac 12\beta V}    =       -\Delta     +     \frac{\beta^2}{4}|\nabla V|^2     -\frac{\beta}2 \Delta V        .  
    }
The latter operator is a specific Schr\"{o}dinger operator.
This operator coincides with the restriction on the level of functions of the 
Witten Laplacian, an operator acting on the full algebra of differential forms, considered by Witten in his celebrated paper~\cite{Witten}. 
Note that the small noise limit $\beta\to 0$  turns now into a problem in semiclassical analysis of Schr\"{o}dinger operators.
In this paper we use two important results, established in this framework: the approximation of the low-lying spectrum
of Schr\"{o}dinger operators via harmonic oscillators sitting at the bottom of the wells \cite{Si},
 and the fine analysis of the splitting of the exponentially small eigenvalues provided in \cite{HKN}. 
For simplicity we restrict here to the case of $\bb R^d$ or of a compact Riemannian manifold as state space, but our arguments could be 
adapted without much difficulty to the case of a bounded domain with reflecting (Neumann) boundary conditions, as considered e.g. in \cite{PSV},
by using~\cite{LP} instead of~\cite{HKN}.

\subsection{Metastable dynamical systems}
\label{ss:kifer}
Here we informally describe some motivations to approach metastability via $\Gamma$-convergence of the Fisher information. By the Birkhoff ergodic theorem, time-averages of an observable of an ergodic (deterministic or random) 
dynamical system converge in the long time limit to the average of the same observable with respect to the invariant measure $m$ of the system. Assume that we have a family of ergodic systems indexed by some parameter
$\beta$, that for the sake of simplicity we may take as a real number.   Under very weak and general conditions, see e.g.\ \cite{Ki}, dynamical systems also satisfy a long-time large deviation principle, that we informally write for each \emph{fixed} $\beta$ as
\bel{e:lddef}{
\bb P_\beta \left(\tfrac{1}{T} \int_0^T\!\!\!\!\!ds f(X_s) \sim \int \! d\mu(x) f(x), \, \forall f\right)\simeq \exp\left(-T E_\beta(\mu)\right),
\quad T\gg 1   ,
} 
where the \emph{rate} $E_\beta(\mu)\ge 0$  is a positive real valued function defined on probability measures\footnote{Note that \eqref{e:lddef} makes sense both for random and deterministic systems, as in the latter case the initial distribution induces a probability measure on the trajectories of the system. In other words, for the deterministic systems, the left hand side of \eqref{e:lddef} should be interpreted as the measure of initial data such that $\int_0^T f(X_s)ds \sim \int d\mu(x) f(x)$.
}.
For each fixed $\beta$, we assume that the system has good ergodic properties, i.e. $E_\beta(\mu)=0$ iff $\mu=m_\beta$ is the invariant measure\footnote{For the sake of simplicity we are assuming that at time $0$ we start with distribution $m_\beta$. This assumption is usually not needed in the random case, but crucial for deterministic systems.}, so that \eqref{e:lddef} can be interpreted as a quantitative version of the Birkhoff theorem.

Suppose now that, informally speaking, a limiting behavior takes place as we move the parameter $\beta$ to some limit, say $\beta\to \infty$. Then the invariant measure is expected to converge to some limiting measure, $m_\beta\to m$. Similarly one expects concentration of time averages to take place in the parameter $\beta$ as well, namely
\bel{e:lddefbeta}{
\bb P_\beta\left(\tfrac{1}{T} \int_0^T\!\!\!\!\!ds f(X_s) \sim \int \! d\mu(x) f(x), \, \forall f\right)
\simeq \exp\left(-T\,a_\beta\, I(\mu)\right),
\quad T\gg 1,\beta\gg 1
}
where $a_\beta$ is a sequence of real numbers converging to $+\infty$. It is a general fact, see e.g.\ \cite{Ma}, that the functional $I$ in \eqref{e:lddefbeta} is the $\Gamma$-limit of $\tfrac{1}{a_\beta}E_\beta$. If ergodic properties of the system hold uniformly in $\beta$, one would expect that $I(\mu)=0$ iff $\mu=m=\lim_\beta m_\beta$.

On the other hand, we say that the system exhibits a \emph{metastable behavior} (on the scale $(a_\beta)$) in the limit $\beta\to\infty$, if the function $I(\mu)=\Glim_\beta \tfrac{1}{a_\beta}E_\beta$, vanishes on a set strictly larger than $\{m\}$. In other words, if there is a 'small-but-not-so-small' probability that, for $T$ large and $\beta$ large, the time average of observables does not behave as the spacial average 
with respect to the limiting invariant measure. In such a case, \eqref{e:lddefbeta} states in particular that one needs to consider time scales much longer than $a_\beta$ to observe convergence of the system to the invariant measure. If the system features a metastable behavior, \eqref{e:lddefbeta} is not completely satisfactory. Indeed, one knows a priori that the left hand side should vanish if $\mu\neq m$ (as we sent $T\to \infty$ before $\beta$, and assumed $m_\beta\to m$), but the 
right hand side does not catch the sharp asymptotic if $I(\mu)=0$ (and $\mu\neq m$). In other words, if $\mu\neq m$ and $I(\mu)=0$ \eqref{e:lddefbeta} just states that the left hand side vanishes on a slower scale than $\exp(-a_\beta T)$ but does not quantify this scale.

In this case, one would expect a further non-trivial expansion in the right hand side of \eqref{e:lddefbeta}, say for $T\gg 1,\beta\gg 1$
\bel{e:lddefbeta2}{
\bb P_\beta \left(\tfrac{1}{T}\int_0^T\!\!\!\!\!ds f(X_s) \sim \int \! d\mu(x) f(x), \, \forall f\right)\simeq \exp\left(-T\,a_\beta\,\left( I(\mu)+ \tfrac{1}{b_\beta}J(\mu) \right)\right),
}
where $b_\beta\to\infty$ and $J(\mu)$ is such that $J(m)=0$, $J(\mu)=\infty$ if $I(\mu)>0$, and thus $J(\mu)$ gives the asymptotic value of the left hand side when $I(\mu)=0$ and $\mu\neq m$. Actually, it may happen that for such $\mu$'s, the left hand side in \eqref{e:lddefbeta2} vanishes at different rates, so that a further expansion of the exponent in the right hand side is necessary, until the exponential behavior of the left hand side at each and every point in $\{I(\mu)=0\}$ is characterized, and a complete quantitative version of the ergodic convergence of averaged observables is recovered.

The discussion informally carried over above can be made very precise, since both large deviations and $\Gamma$-convergence always hold along subsequences, and thus \eqref{e:lddef} can actually be used as a rigorous definition of $E_\beta$, and the existence of a non-trivial development by $\Gamma$-convergence \cite[Chap.~1.10]{Br} (or equivalently of multiple large deviations principles\footnote{This should not be confused with what is called \emph{large and moderate deviations} in the literature, which is not related to metastability. Here the observables are fixed once and for all, and not rescaled. And yet they feature multiple large deviations principles.
} as in \eqref{e:lddefbeta2}) as a general definition of metastability. However we refrain from giving a too abstract formulation of these statements that are not instrumental to our results below. 

\

\subsection{Multiscale gradient flows}
\label{ss:multigrad}
The relative entropy $H_\beta(\mu)$ of the probability measure $\mu$ with respect to the reference measure $m=e^{-\beta V}dx$ is defined as
\bel{e:relent}{
H_\beta(\mu):=\int dm\,\varrho \log \varrho, \qquad \mu=\varrho m
}
It is a well-known fact \cite{JKO} that the Fokker--Planck equation 
\bel{e:fke}{
\partial_t u =\Delta u +\beta \mathrm{div}(u\nabla V )
}
can be written as a gradient flow of the relative entropy $H_\beta$ with respect to the $2$-Wasserstein distance on probability measures on $\bb R^d$. Namely a curve $(\mu_\beta(t))$ of probability measures with bounded second moments on $\bb R^d$ satisfies \eqref{e:fke} iff
\bel{e:gf}{
H_\beta(\mu_\beta(0))\ge H_\beta (\mu_\beta(t))+\tfrac 12 \int_0^t \|\partial_s \mu_\beta(s)\|^2 ds +\tfrac{1}{2} \int_0^t |\nabla_{\mc W_2} H_\beta |^2(\mu_\beta(s))ds
}
where $\|\cdot\|$ is the tangent norm on the  $2$-Wasserstein space and, most importantly for us, the metric gradient $|\nabla_{\mc W_2} H_\beta|^2$ coincides with $\beta I_\beta$. Following the approach in \cite{SS}, and dividing both sides of \eqref{e:gf} by $\beta$, one wants to pass to the limit $\beta \to \infty$. Then, using the $\Gamma$-limit result $I_\beta \to I$ in Theorem~\ref{t:main}, one easily gathers that any limit point $(\mu(t))$ of $(\mu_\beta(t))$ satisfies
\bel{e:gf2}{
\int V d\mu(0) \ge \int V d\mu(t) +\tfrac 12 \int_0^t \|\partial_s \mu(s)\|^2 ds +\tfrac{1}{2} \int_0^t \int |\nabla V|^2 d\mu(s)ds,
}
provided the initial data satisfies $\tfrac{1}{\beta}H_\beta(\mu_\beta(0))\to  \int V d\mu(0)$. In other words, one recovers the simple fact that the solution $u_\beta$ to \eqref{e:fke} satisfies $u_\beta(t/\beta,x)\to u(t,x)$, where $u$ solves the transport equation
\bel{e:fke2}{
\partial_t u = \mathrm{div}(u\nabla V )  .
}
Of course, if the inital data concentrates on critical points of $V$, \eqref{e:fke2} becomes trivial, namely $\partial_t u=0$. Then the development by $\Gamma$-convergence gives the sharp asymptotic of $u_\beta(t/\beta^2)$ exactly in these cases. It turns out that, when the initial condition concentrates on local minima, even this limit is trivial, and still by the $\Gamma$-development one can establish the limit of $u_\beta(t/(\beta e^{\beta W_k}),x)$ where the $W_k$ are appropriate constants defined in \ref{a:5} below.
We do not pursue these problems here, since we plan a detailed study in the non-reversible  case (namely when $\nabla V$ in \eqref{e:fke2} is replaced by a generic vector field) in a follow-up paper. See \cite{PSV} for a detailed discussion on a similar model.
\subsection{Plan of the paper}
\label{ss:plan}
In Section~\ref{s:main} we introduce some definitions and state the main result.
In Section~\ref{s:fisher}, we shortly review some basic properties of the Fisher information functional. In Section~\ref{s:generator} we 
consider suitable quasimodes associated to the generator $L_\beta$. In Section~\ref{s:gammauno} we prove the $\Gamma$-convergence and equicoercivity results for the Fisher information.  In Section~\ref{s:gammadue} and Section~\ref{s:gammatre} we prove the development by $\Gamma$-convergence of the Fisher information 
respectively under inverse power and exponential rescaling.

\section{Main result}
\label{s:main}

\subsection{Basic notation}
\label{ss:notation}
Hereafter $(M,g)$ is a smooth $d$-dimensional Riemannian manifold without boundary and with metric tensor $g$. We assume that either $M$ is compact or that $M=\bb R^d$ and $g$ is the canonical Euclidean tensor. $\Omega$ is the set of smooth, compactly supported $1$-forms over $M$. For $\psi \in \Omega$, $\psi^2$ is 
understood as $\langle\psi,\psi\rangle_g$, while $\nabla$, $\nabla \cdot$, $\hess$ and $\Delta$ are the covariant gradient, the divergence, the Hessian
and the (negative) Laplace-Beltrami operators on $M$. $\mc P(M)$ is the set of Borel probability measures on $M$, and it is naturally equipped with the weak topology of probability measures. If $M$ is compact this is the weakest topology such that the map $\mu\mapsto \int d\mu\,f$ is continuous for all $f\in C(M)$. If $M$ is not compact, the definition is slightly more involved, as one should require $f$ to be bounded and uniformly continuous with respect to any totally bounded distance on $M$ (this topology is independent of such a distance), see \cite[Chap.~3.1]{Str}. In any case, $\mc P(M)$ is a Polish space (that is a completely metrizable, separable topological space), and it is compact if $M$ is compact.

\subsection{$\Gamma$-convergence}
\label{ss:gammaconv}
In this section we briefly recall the basic definitions related to $\Gamma$-convergence, see \cite[Chap.~1]{Br}. For $X$ a Polish space (we will always consider $X=\mc P(M)$) and $(H_\beta)$ a family of lower semicontinuous functions on $X$ indexed by the directed parameter $\beta$ (we will always consider $\beta>0$), one defines
\bel{e:gammadef}{
(\Gliminf_{\beta} H_\beta)(x):=  \inf \left\{ \varliminf_{\beta} H_\beta(x_\beta),\,
                   \{x_\beta\} \subset X\,:\: x_\beta \to x \right\}    ,
\\
\big(\Glimsup_{\beta} H_\beta \big)\, (x):=  
\inf \left\{\varlimsup_{\beta}  H_\beta(x_\beta),\,
                   \{x_\beta\} \subset X\,:\: x_\beta \to x \right\}.
}
Whenever $\Gliminf_\beta H_\beta =\Glimsup_\beta H_\beta=:H$ we say
that $H_\beta$ $\Gamma$-converges to $H$ in $X$. Equivalently,
$H_\beta$ $\Gamma$-converges to $H$ iff for each $x\in X$:
\begin{itemize}
\item[\rm{--}]{for any sequence $x_\beta\to x$ it holds
$\varliminf_\beta H_\beta(x_\beta)\ge H(x)$;}
\item[\rm{--}]{ there exists a sequence $x_\beta\to x$ such that
$\varlimsup_\beta H_\beta(x_\beta)\le H(x)$.}
\end{itemize}
$(H_\beta)$ is \emph{equicoercive} if for each $M\in \bb R$ the set $  \left\{\varlimsup_\beta H_\beta \le M \right\}$ is precompact in $X$.  If $H_\beta$ is equicoercive and $\Gamma$-converges to $H$, then $A:=\mathrm{argmin}(H)$ contains each limit point of $\mathrm{argmin}(H_\beta)$. Denote by $A_0$ the set of such limit points. If $A\setminus A_0$
 is nonempty, there exists a diverging sequence $(a_\beta^{(1)})$ such that the functional
\bel{e:huno}{
H_\beta^{(1)}:= a_\beta^{(1)}(H_\beta-\inf_{x\in X} H_\beta(x))
}
admits a non-trivial (that it having somewhere a value different from $0$ and $\infty$) $\Gamma$-liminf. 
Note that $H^{(1)}_\beta$ inherits from $H_\beta$ the property of being equicoercive, and that $(\Glimsup_\beta H^{(1)}_\beta)(x)=(\Gliminf_\beta H^{(1)}_\beta)(x)=0$ if $x\in A_0$ and $(\Glimsup_\beta H^{(1)}_\beta)(x)=(\Gliminf_\beta H^{(1)}_\beta)(x) = +\infty$ if $x\not\in A$. If $H_\beta^{(1)}$ admits a $\Gamma$-limit $H^{(1)}$ then we say that the development by $\Gamma$-convergence
\bel{e:gdev}{
H_\beta\sim H+\tfrac{1}{a_\beta^{(1)}}H^{(1)}
}
holds. Iterating the procedure with sequences such that $a_\beta^{(k)}/a_\beta^{(k+1)}\to 0$, we say that the development by $\Gamma$-convergence is \emph{full} if it holds
\bel{e:gdevco}{
H_\beta\sim H+\sum_{k} \tfrac{1}{a_\beta^{(k)}}H^{(k)}
}
and for each $x\in A\setminus A_0$ there exists $k$ such that $H^{(k)}(x)\in (0,\infty)$.

Equicoercivity and $\Gamma$-convergence of a sequence
$(H_\beta)$ imply an upper bound of infima over open sets, and
a lower bound of infima over closed sets, see e.g.\
\cite[Prop.~1.18]{Br}, and it is a relevant notion of
variational convergence for the problems discussed in the introduction. A full development by $\Gamma$-convergence then gives a sharp asymptotic of each of such infima.

\subsection{Assumptions on the potential}
\label{ss:assumptions}
The \emph{potential} $V$ is a given real-valued function and we consider the following assumptions.
\begin{enumerate}[label=\textbf{A.\arabic*}]%,ref=A.\arabic*]
\item {\label{a:1}
	 $V\in C^\infty(M)$ is a Morse function. Namely the Hessian of $V$ is non-degenerate at any critical point of $V$.
}
\item {\label{a:2}
In the case $M=\bb R^d$, it holds for some $\beta_0 \ge 0$
\bel{e:a2}{
	\lim_{x\to\infty} \inf_{\beta\ge \beta_0} (\frac14|\nabla V|^2-\frac{1}{2\beta} \Delta V)(x)=+\infty.
}
}
\item {\label{a:3} In the case $M=\bb R^d$, it holds for some $\beta_0\ge 0$ and  all $\beta\ge \beta_0$
\bel{e:a2b}{ 
	Z_\beta :=  \int e^{-\beta V(x)} dx <+\infty.
}
}
\end{enumerate}
\begin{remark}
\label{r:helff}
If $M=\bb R^d$, \ref{a:1}-\ref{a:3} imply that  $\lim_{|x|\to +\infty} V(x)=+\infty$ (see \cite[Proposition~2.2]{HKN}). Thus, in both cases of $M$ compact and $M=\bb R^d$, the set $\{V\le c\}$ is compact for all $c\ge \inf_{x\in M} V(x)$.
\end{remark}
\begin{remark}
Assumption~\ref{a:2} says that, after the ground state transformation of the diffusion generator $L_\beta$ and division by $\beta^2$ (see~\eqref{e:generatorL},\eqref{e:Witten}), the Schr\"{o}dinger potential grows to infinity at infinity, uniformly in $\beta$.
\end{remark}
Let $\wp$ be the set of critical points of $V$, namely $z\in \wp$ iff $\nabla V(z)=0$.
By  \ref{a:1}-\ref{a:2}, $\wp$ is finite.
Let  $\wp_0\subset \wp$ be the set of local minima of $V$ and 
define $W \colon \wp_0 \to (0, \infty]$ as
\bel{e:fw2}
{
W(x): =\inf_{y\neq x\,:\: V(y)\le V(x)}    \inf_{\gamma\in \Gamma(x,y)} \sup_{t\in [0,1]} V(\gamma(t))-V(x)    ,  
}
where $\Gamma(x,y)$ is the set of continuous curves connecting $x$ and $y$ in time $1$. 
In other words $W(x)$ is the lowest mountain pass to climb when going from $x$ to a deeper minimum
of $V$. 

\begin{enumerate}[label=\textbf{A.\arabic*}]%,ref=A.\arabic*]
	\addtocounter{enumi}{3}

\item {\label{a:4} 
 For each $x\in \wp_0$ such that $W(x)<+\infty$ there exists a unique point $\hat x\in \wp$ such that the two following conditions hold:
\begin{itemize}
\item[(i)]{ $V(\hat x)=V(x)+W(x)$.}
\item[(ii)]{ $\hat x$ and $x$ lie in the same connected component of the compact set $\left\{y\in M\,:\: V(y)\le V(x)+W(x) \right\}$.}
\end{itemize}
In other words, there is a unique saddle point $\hat x$ such that all the optimal curves $\gamma$ in the variational problem \eqref{e:fw2} pass through $\hat x$.
}

\item {\label{a:5}
	$W(x)\neq W(y)$ whenever $x\neq y$, $x,\,y\in \wp_0$.
}
\end{enumerate}	
By \ref{a:5}  we can label $\wp_0=\{x_0,\ldots,x_n\}$ by requiring 
 \bel{e:order}{
W_n<\ldots< W_1 <W_0=+\infty    ,
 }
where we use the shorthand notation $W_i := W(x_i)$.  
Note that $x_0$ is the unique global minimizer of $V$ and  $\hat x_i$
exists for $i\neq 0$ by \ref{a:4}.

\subsection{The $\Gamma$-development theorem}
\label{ss:theorem}
For $\beta\ge \beta_0$ define the reference measure $m_\beta$ on $M$ as 
\bel{e:refm}{
dm_\beta(x)=e^{-\beta V(x)} dx,
}
where $dx$ is the Riemannian volume on $M$. The \emph{Fisher information} at inverse temperature $\beta>0$ is the functional (see Definition~\ref{d:dvmore} and \eqref{e:ibeta2} for 
further details)
\bel{e:ibeta}
{
& I_\beta \colon \mc P(M)\to [0,+\infty]   , 
\\
& I_\beta(\mu):= 
\begin{cases}
\frac{1}{2\beta^2} \int dm_\beta \frac{|\nabla \varrho|^2}{\varrho}& 
 %\text{if $d\mu = \varrho dm$ and $\nabla \log \varrho \in L^2(\mu)$}
\text{if $\frac{d\mu}{dm_\beta}=\varrho$;}
\\
+\infty & \text{otherwise}. 
\end{cases}
}
We are interested in the variational convergence of $I_\beta$ in the low temperature regime, namely when $\beta\to \infty$. In order to describe our main result we need to define further notation.

For $z\in \wp$, we denote by $(\xi_i(z))_{i=1,\ldots,d}$ the eigenvalues of $\hess V(z)$, labelled by ordering $\xi_1(z)\le \xi_2(z)\le \ldots\le \xi_d(z)$. Then define
\bel{e:zeta}{
& \zeta\colon M\to [0,+\infty]
\\
&\zeta(z):=
\begin{cases}
 \sum_{i=1}^d |\xi_i(z)| - \xi_i(z)=2 \sum_{i\,:\: \xi_i(z)<0}|\xi_i(z)| & \text{if $z\in \wp$;}
 \\
 +\infty &\text{otherwise.}
 \end{cases}
}
For $k\ge 1$ and $x_k\in \wp_0$, the saddle point $\hat x_k$ defined in \ref{a:4} is easily seen to satisfy $\xi_1(\hat x_k)<0$ and $\xi_i(\hat x_k)>0$ for $i\ge 2$. Recalling the labelling \eqref{e:order}  we define for $k=1,\ldots,n$
 \bel{e:coeffk}{
& \eta_k\colon M\to [0,+\infty]
\\
& \eta_k(x):=
\begin{cases}
0 & \text{if $x\in \{x_0,\ldots,\,x_{k-1}\}$;}
\\
 \frac{|\xi_1(\hat x_k)|}{\pi}\sqrt{ \frac{\det (\hess V)(x_k)}{|\det (\hess V)(\hat x_k)|}} =\sqrt{ \frac{|\xi_1(\hat x_k)| \prod_{i=1}^d \xi_i(x_k) }{\pi^2 \prod_{i=2}^d  \xi_i(\hat x_k)}} & \text{if $x=x_k$;}
 \\
 +\infty &\text{otherwise.}
 \end{cases}
}
The functions $\zeta$ and $\eta_k$ are lower semicontinuous and coercive. Their expression 
%numbers in their definition, involving the quadratic approximation of
%$V$ around critical points, 
naturally appears when studying the asymptotic behaviour of eigenvalues of $L_\beta$ (see Theorem~\ref{t:simon} and
~\ref{t:HKN} for this interpretation). 
Our main result is the following.
\begin{theorem}
\label{t:main}
Assume \ref{a:1}-\ref{a:5}. Then $(I_\beta)$ is an equicoercive sequence of lower semicontinuous functionals. Moreover
\bel{e:gc1}
{
\Glim_\beta I_\beta= I ,
}
\bel{e:gc2}{
\Glim_\beta \beta I_\beta =J,
}
\bel{e:gc3}
{
\Glim_\beta (\beta\,e^{\beta W_k} I_\beta)= J_k \qquad    \text{ for } k=1,\dots,n ,
}
where the lower semicontinuous, coercive functionals $I,\,J,\,J_k\colon \mc P(M)\to [0,+\infty]$ are defined as
\bel{e:i}
{
I(\mu):=    \tfrac 12 \int d\mu(x)\,\left|\nabla V\right|^2(x).
}
\bel{e:j}
{
J(\mu)=\int d\mu(x) \,\zeta(x)=
\begin{cases}
\sum_{y\in \wp} \alpha_y \zeta(y)
			& \text{if $\mu=\sum_{y\in \wp} \alpha_y \delta_{y}$;}
\\
+\infty & \text{otherwise.}
\end{cases}
}
\bel{e:jk}
{
J_k(\mu)=\int d\mu(x)\,\eta_k(x)
=
\begin{cases}
\alpha_k \,\eta_k(x_k) & \text{if $\mu=\sum_{j=0}^k \alpha_j\delta_{x_j}$;}
\\
+\infty & \text{otherwise.}
\end{cases}
}
\end{theorem}
%Some remarks are in order.
\begin{remark}
\label{r:gammadev}
It holds $\inf_{\mu} I(\mu)=\inf_{\mu} J(\mu)=\inf_{\mu} J_k(\mu)=0$. The number of minimizers of the functionals in the sequence $I,J,J_n,\dots, J_1$ decreases from $|\wp|$ to~$1$. More precisely we have
\begin{itemize} 
 \item $I(\mu)=0$ iff $J(\mu)<+\infty$,
 namely iff $\mu$ is concentrated on $\wp$. 
 \item $J(\mu)=0$ iff $J_n(\mu)<+\infty$, namely iff $\mu$ is concentrated on $\wp_0$. 
 \item for $k=1,\ldots,n$-1, $J_{k+1}(\mu)=0$ iff $J_k(\mu)<+\infty$, and finally $J_1(\mu)=0$ iff $\mu=\delta_{x_0}$ (that is the unique limit point of $m_\beta$, the minimizer of $I_\beta$). 
 \end{itemize}
 In other words, 
 $I_\beta$ admits the following full development by $\Gamma$-convergence, see \cite[Chap.~1.10]{Br},  as $\beta \to \infty$:
\bel{e:dgc}
{
I_\beta(\mu)\sim I(\mu)+ \frac{1}{\beta} J(\mu)+\sum_{k=1}^n \frac{1}\beta\,e^{-\beta\,W_k}J_k(\mu)  . 
}
In particular the 
 asymptotic behavior of infima of $I_\beta$ on closed and open subsets is characterized by a finite and non-zero leading order, see \cite[Theorem~1.18]{Br}.   \end{remark}
\begin{corollary}     
\label{c:ld}
Consider the diffusion process $X^\beta$ defined through
\bel{e:diffusion}{
&\dot X^\beta_t= - \nabla V(X^\beta_t) +\sqrt{2\beta^{-1}} \dot W_t 
\\
&X^\beta(0)=x_0 .
}
Equivalently, let $X^\beta$ be the Markov process with generator $\beta^{-1} \Delta-\nabla V \cdot \nabla$, and starting at $x_0\in M$. For $\beta,\,T>0$, consider the empirical measure
\bel{e:emp}{
\theta_{\beta,\,T}:=\frac{1}{T}\int_0^T \delta_{X^\beta_t}dt\in \mc P(M).
}
In other terms, $\theta_{\beta,T}$ is the random probability measure on $M$ such that for all test functions $\varphi$ it holds $\int d\theta_{\beta,T}(\varphi)=\tfrac{1}{T} \int_0^T \varphi(X^\beta_t)dt$. Then, in the sense of large deviations \cite{DZ}, and in the limit $T>>\beta\to +\infty$, $\theta_{\beta,T}$ satisfies (independently of $x_0$)
\bel{e:ld1}{
\bb P(\theta_{\beta,T}\sim \mu)\sim e^{-\beta T\left(I(\mu)+\tfrac{1}{\beta} J(\mu)+\sum_{k=1}^n \frac{e^{-\beta\,W_k}}{\beta} J_k(\mu) \right)}.
}
More precisely, for all $I$-regular (respectively $J$-regular and $J_k$-regular) sets $\mc A\subset \mc P(M)$ it holds
\bel{e:ld2}{
\lim_{\beta\to \infty}\lim_{T\to \infty}  \frac{1}{\beta \,T}\log \bb P(\theta_{\beta,T}\in \mc A)= -\inf_{\mu\in \mc A} I(\mu),
}
\bel{e:ld3}{
\lim_{\beta\to \infty} \lim_{T\to \infty} \frac{1}{T} \log \bb P(\theta_{\beta,T}\in \mc A)= -\inf_{\mu\in \mc A} J(\mu),
}
\bel{e:ld4}{
\lim_{\beta\to \infty} \lim_{ T\to \infty} \frac{e^{\beta W_k} }{T} \log \bb P(\theta_{\beta,T}\in \mc A)= -\inf_{\mu\in \mc A} J_k(\mu).
}
\end{corollary}
\begin{proof}
It is not hard to check that, under \ref{a:1}-\ref{a:3}, the hypotheses in \cite{DV} are satisfied, and thus for each fixed $\beta>0$ the empirical measure $\theta_{\beta,\,T}$ satisfies a large deviations principle with rate $\beta\,I_\beta$ as $T\to +\infty$. As well known and easy to prove, the rate function of the large deviations with speed $a_\beta \,T$, for the directed familiy $\theta_{\beta,\,T}$ as $T\to \infty$ and $\beta \to \infty$ is then given by the $\Gamma$-limit of $\frac{1}{a_\beta}\beta I_\beta$, see \cite[Corollary~4.3]{Ma}. Thus the statement follows from Theorem~\ref{t:main} when using respectively $a_\beta= \beta$, $a_\beta=1$ and $a_\beta=e^{-\beta W_k}$.
\end{proof}

\section{The Fisher information}
\label{s:fisher}
In this section we shortly collect some basic facts concerning the Fisher information, we also refer to \cite[Definition~20.6]{Vi} for further details. We only sketch the proofs here since we just restate some facts in what is a well understood framework. In this section we omit the dependence on $\beta$ for the sake of readability.
\begin{definition}
\label{d:dvmore}
Let $V\in C^0(M)$. The Fisher information $I \colon \mc P(M)\to [0,+\infty]$ 
induced by $V$ is defined as
\bel{e:ex1}{
I(\mu) &
    = \sup_{\psi\in \Omega}
    \int d\mu\, \left( e^V \nabla \cdot \psi -\tfrac 12 \,e^{2V} \psi^2\right).
    }
\end{definition}
\noindent
Define the measure $m\in \mc P(M)$ as $dm(x)=e^{-V(x)}dx$.
\begin{remark}
\label{r:lsc}
As a supremum of linear and continuous functionals, $I$ is convex and lower semicontinuous. Moreover if $V\in C^2(M)$, the change of variable $\psi \to -e^{-V}\,\psi$ shows
\bel{e:ex2}{
I(\mu) &  =  \sup_{\psi\in \Omega}
    \int d\mu \left(\nabla V\cdot \psi -\nabla \cdot \psi - \tfrac 12   \,|\psi|^2\right)
    \\ & =
      \begin{cases}
    \tfrac 12 \int dm \tfrac{|\nabla \varrho|^2}{\varrho} & \text{if $d\mu = \varrho dm$ and $\nabla \log \varrho \in L^2(\mu)$}
    \\
  +\infty & \text{otherwise}
   \end{cases}
   \\ &
   =
      \begin{cases}
      \int dx\, \, 2|\nabla h|^2 + h^2 \left(\frac{|\nabla V|^2}2 -\Delta V\right)& \text{if $d\mu = h^2\,dx$ and $h\in H^1(M)$}
    \\
  +\infty & \text{otherwise.}
   \end{cases}   
}
\end{remark}
Let now
 \bel{e:pitilde}{
 \tilde{\mc P}(M):=\left\{\mu\in\mc P(M)\,:\: \frac{d\mu}{dm}\in C^2(M),\, \exists \eps>0\,:\:\frac{d\mu}{dm}\ge \eps  \right\}.
 }
\begin{proposition}
\label{p:dvmore}
Assume that $\int dx \,e^{-V(x)}<+\infty$ and $V\in C^1(M)$. Then $I$ is the lower semicontinuous envelope of $\tilde{I}\colon \mc P(M)\to [0,+\infty]$ defined as
\bel{e:ex3}{
\tilde{I}(\mu) = 
  \begin{cases}
    \tfrac 12 \int dm \tfrac{|\nabla \varrho|^2}{\varrho} & \text{if $\mu\in \tilde{\mc P}(M)$}
    \\
  +\infty & \text{otherwise.}
   \end{cases}
}
\end{proposition}
\begin{proof}
Since changing $V$ to $V+\log \int dx \,e^{-V(x)}$ does not change $I$, here we assume $\int dx \,e^{-V(x)}=1$. We need to prove that for each $\mu$ there exists $\mu^n\in \tilde {\mc P}(M)$ such that $\mu^n\to \mu$ and $\varlimsup_n \tilde{I}(\mu^n)\le I(\mu)$. Said differently, that $\tilde{\mc P}(M)$ is $I$-dense in $\mc P(M)$. Setting
\bel{e:mundv}{
\mu^n=(1-\tfrac{1}{n}) \mu+\tfrac{1}{n} m
}
one easily gathers $\mu^n\to \mu$ and $I(\mu^n)\le I(\mu)$. Therefore it is enough to prove that $\tilde{\mc P}(M)$ is $I$-dense in the set of $\mu\in \mc P(M)$ with density bounded away from $0$. This is immediate from the density of smooth functions in weighted Sobolev spaces.

%%\comment{Il conto per il sup:
%%\bel{e:excomm}{
%%I(\mu) & =\tfrac 12 \int_M dm(x) \tfrac{|\nabla \varrho|^2}{\varrho}
%%    = \sup_{\psi} \int dm \,\psi\,\nabla \varrho - \tfrac 12 \int dm \,\varrho \,\psi^2
%%    \\ & 
%%    =\sup_{\varphi=-e^{-V}\psi } -\int dx\,\varphi  \nabla \varrho -\tfrac 12 \int dm \varrho \,e^{2V} \varphi^2
%%   \\ &
%%    =\sup_{\varphi}  \int dx \varrho \nabla \varphi -\tfrac 12 \int d\mu\,e^{2V} \varphi^2
%%         \\ &
%%=       \sup_{\varphi}  \int d\mu\, \left( e^V \nabla \varphi -\tfrac 12 \,e^{2V} \varphi^2\right)
%%}
%%The supremum is achieved for $\psi=\nabla \varrho/\varrho$, that is $\varphi=e^{-V}\nabla \log \varrho$.
%%}
\end{proof}

In view of Remark~\ref{r:lsc}, the functional $I_\beta$, already introduced in~\eqref{e:ibeta} and appearing in our main theorem, can be equivalently defined as
\bel{e:ibeta2}{
I_\beta(\mu):=\beta^{-2}  I^{\beta V} , 
}
where $I^{\beta V}$ is the Fisher information induced by $\beta V$ as defined in Definition~\ref{d:dvmore}.

\section{Spectral analysis of the generator}
\label{s:generator}

In this section we denote by $\dangle{\cdot}{\cdot}$ the inner product in $L^2(m_\beta)$. 
Consider the operator
\bel{e:diffgen}{L_\beta f:=\Delta f-\beta \nabla V \cdot \nabla f
}
defined for $f\in C^\infty_{c}(M)$. By \ref{a:2} and standard results, it uniquely extends to a self-adjoint operator in $L^2(m_\beta)$ with compact resolvent. We still denote by $L_\beta$ such an extension.

\subsection{Quasimodes of the bounded spectrum}
\label{ss:simon}
For $z\in \wp$ and $\undn=(n_1,\ldots,n_d)\in \bb N^d$ let
\bel{e:deflamb}{
\lambda_{z,\undn}:=   \frac {\zeta(z)}{2}+   \sum_{k=1}^d n_k |\xi_k(z)|,
}
where  $\zeta$ and the $\xi_k$'s are defined in Subsection~\ref{ss:theorem}. 
The set $(\lambda_{z,\undn})_{\undn\in \bb N^d}$ is the spectrum of a suitably rescaled quadratic approximation (i.e.\ of a harmonic oscillator) around the point $z\in \wp$ of 
the   Schr\"{o}dinger operator  \eqref{e:Witten}, see also the proof of Lemma~\ref{l:pavese} for this interpretation. Define also
\bel{e:lambdacard}{
S_\lambda:=\left\{(z,\undn)\in \wp \times \bb N^d\,:\:\lambda_{z,\undn}
	=\lambda\right\}.     
}

The following theorem conveniently resumes in our setting some of the main results obtained in the paper \cite{Si}          
 in a more general framework of semiclassical Schr\"{o}dinger operators. Throughout this Subsection~\ref{ss:simon} it is sufficient to assume just \ref{a:1} and \ref{a:2}.
 %under weaker hypotheses than \ref{a:1}-\ref{a:5}.
% \begin{theorem}
% \label{t:simon}
%For $z\in \wp$ and $\undn=(n_1,\ldots,n_d)\in \bb N^d$ let
%$\lambda_{z,\undn}:=   \frac {\zeta(z)}{2}+   \sum_{k=1}^d n_k |\xi_k(z)|$ and fix $M>\sup_{z\in \wp} \lambda_{z,\underline 0}$. There exists $\beta_0:=\beta_0(M)>0$ and for $z\in \wp$ and $\undn\in \bb N^d$ an $\eps_{z,\undn}>0$ such that for each $\beta \ge \beta_0$
%\bel{e:specs}{
%\mathrm{Spec}(-\tfrac{1}{\beta} L_\beta)\cap [0,M]\subset \bigcup_{z\in \wp,\, \undn\in \bb N^d} (\lambda_{z,\undn}-\eps_{z,\undn},\lambda_{z,\undn}+\eps_{z,\undn})   \cap [0,M]
%}
%Moreover for each $z\in \wp$ and $\undn \in \bb N^d$ such that $\lambda_{z\undn}\le M$, the cardinality of $\mathrm{Spec}(-\tfrac{1}{\beta} L_\beta)\cap (\lambda_{z,\undn}-\eps_{z\undn}, \lambda_{z,\undn}+\eps_{z\undn})$ is independent of $\beta\ge \beta_0$.
%\end{theorem}
 \begin{theorem}[Bounded eigenvalues]
 \label{t:simon}
 Fix $\Lambda>0$. There  exists $\eps_\Lambda>0$ such that for each $\eps \in (0,\eps_\Lambda)$
 the following holds:
  there exists $\beta_{\eps}>0$ so that for $\beta\ge \beta_{\eps}$ 
\bel{e:lambdaa}{
\mathrm{Spec}(-\tfrac{1}{\beta} L_\beta)\cap [0,\Lambda]    \ \subset     \  \bigcup_{\lambda\,:\:S_\lambda\neq \emptyset} (\lambda-\eps,\lambda+\eps)\cap [0,\Lambda]. 
%\cup_{z\in \wp,\undn\in \mathbb N^d} (\lambda_{z,\undn}-\eps,\lambda_{z,\undn}+\eps)
}
Moreover, for each $\lambda\in [0,\Lambda]$ such that $S_\lambda\neq \emptyset$, the cardinality of $\mathrm{Spec}(-\tfrac{1}{\beta} L_\beta)\cap (\lambda-\eps,\lambda+\eps)$ equals, when each eigenvalue is counted with its multiplicity, the cardinality of $S_\lambda$.
\end{theorem}  
This theorem states that in the limit $\beta\to \infty$, counting multiplicity, the spectrum of $-\tfrac{1}{\beta} L_\beta$ is well-aproximated 
by the spectrum of the direct union of suitable harmonic oscillators. It is necessary to fix a 
threshold $\Lambda$, since this harmonic approximation is not uniform
for $\Lambda\to \infty$. 

%
%for each $(z,\undn)$ and for $\beta$ large enough, there is an eigenvalue of $-\tfrac{1}{\beta} L_\beta$ close to $\lambda_{z,\undn}$ (eigenvalues associated to different $(z,\undn)$'s may coincide giving rise to eigenvalues with higher multiplicity). The statement of the theorem is slightly more involved since convergence of eigenvalues is not uniform in the size of the eigenvalues themselves.

For $\lambda\geq 0$ and $\eps>0$ denote by 
\bel{e:proj2}{
P_{\beta,\eps,\lambda}:= \un { (\lambda-\eps, \lambda+\eps)}(-\tfrac{1}{\beta}L_\beta)
}
 the spectral projection of $-\tfrac{1}{\beta}L_\beta$ associated to the interval $ (\lambda-\eps, \lambda+\eps)$.

\begin{proposition}
\label{p:vittorini}
There exists an orthonormal base $(\Psi_{\beta,z,\undn})_{z\in \wp,\undn \in \bb N^d}$ of $L^2(m_\beta)$ such that for every $\eps>0, z\in \wp
$ and $ \undn \in \bb N^d$ the following holds: 
there exists a $\beta_{\eps, z, \undn }>0$ such that for $\beta> \beta_{\eps, z, \undn}$
and for every $\varphi\in C_b(M)$,
\bel{e:eigen}{
\Psi_{\beta,z,\undn} \in \mathrm{Range}\left(P_{\beta,\eps,\lambda_{z,\undn}}\right)    ,
}
%%\bel{e:lambdaconv}{
%%\lim_\beta \lambda_{\beta,x,\undn}=\lambda_{x,\undn}
%%}
\bel{e:measconv2}{
\lim_\beta      \int \Psi_{\beta, z, \undn} \,\Psi_{\beta, z', \undn'}\, \varphi\, dm_\beta    =  
\begin{cases}\varphi(z)    &     \text{ if   $z=z'$ and   $\undn=\undn' $}   \\
0   &    \text{ otherwise}
\end{cases}.
}
\end{proposition}
We stress that
%, as opposed to the case of exponentially small eigenvalues of $L_\beta$ discussed in Section~\ref{ss:HKN},
eigenbases of $L_\beta$ in general fail to satisfy \eqref{e:eigen}-\eqref{e:measconv2}, due to possible resonances (that is $\lambda_{z,\undn}=\lambda_{z',\undn'}$ for some $z\neq z'$ and $\undn$, $\undn'$). We prove the proposition after the preliminary Lemma~\ref{l:pavese}. 

\begin{lemma}
\label{l:pavese}
 For each $z\in \wp$, $\undn\in \bb N^d$ there exists $\tilde \Psi_{\beta, z, \undn}\in C^2_c(M)$
such that for every $\varphi\in C_b(M)$
\bel{e:lemmai}{
  \lim_{\beta} \int  |(-   \tfrac 1\beta  L_\beta- \lambda_{z,\undn}) \tilde \Psi_{\beta, z, \undn} (x)    |^2   dm_\beta
=   0 ,
}
\bel{e:lemmaii}{
\lim_\beta      \int \tilde \Psi_{\beta, z, \undn}\tilde \Psi_{\beta, z', \undn'} \varphi\, dm_\beta    =  
\begin{cases}\varphi(z)    &     \text{ if   $z=z'$ and   $\undn=\undn' $}   \\
0   &    \text{ otherwise}
\end{cases}.
}
\end{lemma}
\begin{proof}
First consider the case $M=\bb R^d$.  Define
 for $\beta>0$ the differential operator  (see also~\eqref{e:Witten})
\bel{e:abeta}{
A_\beta  :=     -    \Delta    +   \frac{1}{4}  \beta^2 |\nabla V|^2  -     \frac{1}{2} \beta \Delta V   
}
and let
 for $z\in \wp$ 
 \bel{e:hbetaz}{
  H_{\beta,z} :=  - \Delta    +      \beta^2 Q_z   -   \beta  C_z    ,
  }
where
\bel{e:qzz}{   Q_z(x)      =      \frac 14 \hess^2 V(z) (x-z)\cdot(x-z) ,  \qquad   C_z   =   \frac 12\mathrm{Trace} (\hess V(z)) .
}
The operator $H_{\beta,z}$ extends to a self-adjoint operator in $L^2(dx)$, which is a harmonic oscillator shifted by the constant $\beta C_z$. In particular its spectrum is given by $(\beta \lambda_{z, \undn})_{\undn\in \bb N^d}$. Let now $(\Theta_{z,\undn})_{\undn\in \bb N^d}$ be a corresponding orthonormal basis of eigenfunctions of $H_{z, \beta=1}$; the $\Theta_{z,\undn}$ are Hermite functions centered at $z$. It is easy to check that $(\Theta_{\beta,z,\undn})_{\undn}$ defined by
\bel{thetazeta}{
\Theta_{\beta,z,\undn}(x)=\beta^{1/4} \Theta_{z,\undn}(z+\sqrt{\beta}(x-z))
}
is an orthonormal base of eigenfunctions of $H_{z,\beta}$ and
\bel{e:convtheta}{
\lim_{\beta\to+\infty} \Theta_{\beta,z,\undn}^2 dx=\delta_z(dx) \qquad \text{in $\mc P(M)$.}
}
Then define 
$\tilde \Theta_{\beta, z, \undn } (x):=
\frac{\chi_{z}(x)  \Theta_{\beta,z,\undn} ( x) }{ \sqrt {Z_{\beta,z, \undn}}   }$, where
$\chi_z\in C_c^\infty(M;[0,1])$ is a smooth cut-off function such that $\chi\equiv 1$ on a ball $B_r(z)$ centered at $z$ of radius $r$, and  $\chi\equiv 0$ on $B^c_{2r}(z)$, where $r>0$ is chosen sufficiently small (in particular such that the $\chi_z$'s have pairwise disjoint supports). 

It follows from Taylor expansion of the Schr\"{o}dinger potential $ \frac{1}{4}  \beta^2 |\nabla V|^2   -      \frac{1}{2} \beta \Delta V   $  that, as $x \to z$,
\bel{e:abetastima}{
 A_\beta =   H_{\beta,z}         +    O\left(\beta^2 |x-z|^3\right)   +   O\left(\beta|x-z|\right)  .
 }
%and
%\[    \sqrt{Z_{\beta,z, \undn} }    \Delta\tilde \Psi_{\beta, z, \undn } (x)     =    
%      \chi_{z}(x)     \Delta \Theta_{z,\undn} ( x)    +      
%        \Theta_{z,\undn} ( x)    \Delta\chi_{z}(x)      +
%    2 \nabla \chi_{z}(x)  \cdot \nabla \Theta_{z,\undn} ( x)            ,         \]
Thus one obtains
\bel{e:abetastimaz}{
& \int \left[(A_\beta -  \beta \lambda_{z, \undn})
  \tilde \Theta_{\beta, z, \undn } (x)\right]^2 dx     =  
 \\  
& \qquad   \frac{1}{Z_{\beta,z, \undn}  }   \int     \Big[
    \Theta_{\beta,z,\undn}     \Delta\chi_{z} 
+   2 \nabla \chi_{z}  \cdot \nabla \Theta_{\beta,z,\undn} 
\\ & \phantom{  \qquad \frac{1}{Z_{\beta,z, \undn}  }   \int     \Big[}
  +   \left(O\left(\beta^2 |x-z|^3\right)  
    +   O\left(\beta|x-z |\right)   \right)
  \chi_{z}  \Theta_{\beta,z,\undn}     \Big]^2  dx       .  
}
 Hence, using the explicit form of the $\Theta_{\beta,z,\undn}$'s as Hermite functions centered at $z$, it follows by Laplace asymptotics that
 \bel{e:normaH2}{    \tfrac{1}{\beta^2} \int     \left[(A_\beta -  \beta \lambda_{z, \undn}) \tilde \Theta_{\beta, z, \undn } (x)\right]^2 dx    =    O(\beta^{-1})  . }
Moreover, from \eqref{e:convtheta} and $L^2(dx)$ orthogonality of the Hermite functions $(\Theta_{z,\undn})_{\undn}$ it easily follows
for $\varphi\in C_b(M)$
\bel{e:thetadelta}{ \lim_\beta      \int \tilde \Theta_{\beta, z, \undn}\tilde \Theta_{\beta, z', \undn'} \varphi\, dx   =  
\begin{cases}\varphi(z)    &     \text{ if   $z=z'$ and   $\undn=\undn' $}   \\
0   &    \text{ otherwise}
\end{cases}.
}
Finally define $\tilde\Psi_{\beta, z, \undn } =  \tilde \Theta_{\beta, z, \undn } e^{\frac{\beta V}{2}}$ and 
note that
\[     \int  |(-   \tfrac 1\beta L_\beta- \lambda_{z,\undn}) \tilde \Psi_{\beta, z, \undn} (x)    |^2   dm_\beta
=      \tfrac{1}{\beta^2} \int  |(A_\beta- \beta\lambda_{z,\undn}) \tilde \Theta_{\beta, z, \undn} (x)    |^2   dx          \]
and
\[   \int \tilde \Psi_{\beta, z, \undn}\tilde \Psi_{\beta, z', \undn'} \varphi\, dm_\beta    =  
  \int \tilde \Theta_{\beta, z, \undn}\tilde \Theta_{\beta, z', \undn'} \varphi\, dx 
. 
     \]
Therefore \eqref{e:lemmai} and \eqref{e:lemmaii} follow from \eqref{e:normaH2} and \eqref{e:thetadelta}.
   
The proof in the case  of $M$ a compact manifold follows from a straightforward adaptation of the previous arguments: since the Laplace-Beltrami operator in coordinates
takes the form 
\bel{e:laplaceBeltrami}{    \sum_{i,j} g^{ij} \partial_i\partial_j     +    \left(   \partial_j g^{ij}    + g^{ij} \partial_j \log \sqrt{\det g}      \right)   \partial_j     ,   }
using in particular 
local coordinates $\{y= y^{(z)}\}$ around each $z\in \wp$ such that the metric tensor is the identity in $0$, leads (instead of \eqref{e:abetastima}) 
to the local estimate 
\bel{e:abetastimamanifold}{ A_\beta =   H_{\beta,z}   +      \sum_{i}   \left[   O\left(  |y|   \right)   \partial_i\partial_i   +
   O( 1)   \partial_i    \right]             +    O\left(\beta^2 |y|^3\right)   +   O\left(\beta|y|\right)          .         }
Chosing $r$ sufficiently small such that the support of $\chi_z$ is contained in the corresponding local chart $\{y^{(z)}\}$, one can argue as before through Laplace asymptotics
and show that the additional terms appearing in \eqref{e:abetastimamanifold} give again an $O(\beta)$ contribution. 
\end{proof}

\begin{remark}
\label{r:hilbert}
Let $T$ be a bounded operator on a Hilbert space $H$ and let $u,v\in H$ be normalized. Then
\bel{e:tt}{
\left| \dangle{v}{Tv'}-\dangle{u}{Tu'} \right|\le 2 \,\|T\|     \left(  \sqrt{1-\dangle{u}{v}^2}    +  \sqrt{1-\dangle{u'}{v'}^2}   \right)   .
}
\end{remark}
\begin{proof}
Write $v=\dangle{u}{v}u+(v-\dangle{u}{v}u)$ and note that $\|v-\dangle{u}{v}u\|_{H}=1-\dangle{u}{v}^2$. Using the inequality
$\sqrt{1-a^2} \sqrt{1-a'^2} + \sqrt{1-aa'}\le \sqrt{1-a^2} + \sqrt{1-a'^2}$ for $a=\dangle{u}{v}$ and
$a'=\dangle{u'}{v'}$  one gets the result by straigthforward linear algebra. 
%\bel{e:tt1}{
% \dangle{v}{Tv}-\dangle{u}{Tu}  = & \dangle{u}{v}^2\dangle{u}{Tu} - \dangle{u}{Tu} 
% \\ & +
%\dangle{ (v-\dangle{u}{v}u)}{T (v-\dangle{u}{v}u)}+2 \dangle{v}{T (v-\dangle{u}{v}u)}
%\\ \le & 2\,\| T\| (1-\dangle{u}{v}^2+ \sqrt{1-\dangle{u}{v}^2} )
% \le 4 \| T\| \sqrt{1-\dangle{u}{v}^2} 
%}
\end{proof}

\begin{proof}[Proof of Proposition~\ref{p:vittorini}]
We use the same notation as in Theorem~\ref{t:simon}. Fix $\Lambda>0$ large enough. For $\lambda< \Lambda$ such that $S_\lambda\neq \emptyset$,
for $(z, \undn)\in S_\lambda$ take  $\tilde\Psi_{\beta,z,\undn}$ as in Lemma~\ref{l:pavese}. Define $\bar \Psi_{\beta,z,\undn}:=P_{\beta,\eps_{\Lambda},\lambda}\tilde\Psi_{\beta,z,\undn}$ as the spectral projection of $\tilde\Psi_{\beta,z,\undn}$  corresponding to the  interval $(\lambda-\eps_{\Lambda}, 
\lambda+\eps_{\Lambda})$, see \eqref{e:proj2}. 

Now for $(z,\undn),(z',\undn)\in S_\lambda$, since $P_{\beta,\eps_{\Lambda},\lambda}$ is a projector,
\begin{gather}     \dangle{\bar \Psi_{\beta,z,\undn}}{\bar  \Psi_{\beta,z',\undn'}}_{L^2(m_\beta)}  =    \nonumber \\     \dangle{\tilde \Psi_{\beta,z,\undn}}{\tilde  \Psi_{\beta,z',\undn'}}_{L^2(m_\beta)}
-  \dangle{\tilde \Psi_{\beta,z,\undn}   - \bar  \Psi_{\beta,z,\undn}      }{\tilde  \Psi_{\beta,z',\undn'}   -    \bar \Psi_{\beta,z',\undn'} }_{L^2(m_\beta) }   .  \label{e:scal} 
   \end{gather}
The definition of $\bar \Psi_{\beta,z,\undn}$ readily implies
\bel{e:tildebar}{ 
  \|   \tilde \Psi_{\beta,z,\undn}   - \bar  \Psi_{\beta,z,\undn }   \|_{L^2(m_\beta)}  
&
 = \| ( I-P_{\beta,\eps_{\Lambda},\lambda})    \tilde \Psi_{\beta,z,\undn}   \|_{L^2(m_\beta)}  
    \\ &
     \leq    \frac{1}{\eps_{\Lambda}}  
\|(-\tfrac 1 \beta L_\beta- \lambda_{z,\undn}) \tilde \Psi_{\beta, z, \undn}\|_{L^2(m_\beta)}    .
}
In view of \eqref{e:lemmai} this vanishes as $\beta\to\infty$. Therefore the last term in \eqref{e:scal} vanishes as $\beta\to\infty$ and \eqref{e:lemmaii} now gives
\bel{e:psipsi}{   \lim_\beta \dangle{\bar \Psi_{\beta,z,\undn}}{\bar  \Psi_{\beta,z',\undn'}}_{L^2(m_\beta)}     =
\begin{cases}1    &     \text{ if   $z=z'$ and   $\undn=\undn' $}   \\
0   &    \text{ otherwise} .
\end{cases}}    
In particular for $\beta$ large enough the $(\bar \Psi_{\beta, z, \undn})_{(z,\undn)\in S_\lambda}$ are linearly independent, and span 
$\mathrm{Range} (P_{\beta,\eps_{\Lambda},\lambda})$ by Theorem~\ref{t:simon}.
Now construct $(\Psi_{\beta,z,\undn})_{(z,\undn)\in S_\lambda}$ by applying the Gram-Schmidt
procedure on $(\bar \Psi_{\beta,z,\undn})_{(z,\undn)\in S_\lambda}$. From \eqref{e:psipsi} and Lemma~\ref{l:pavese}-(ii)  (with $\varphi\equiv1$) it follows that
\bel{e:gs}{    \lim_\beta    \dangle{ \tilde \Psi_{\beta,z,\undn} }{\Psi_{\beta,z',\undn' } }_{L^2(m_\beta)}    =
\begin{cases}    
1   &     \text{ if   $z=z'$ and   $\undn=\undn' $}   \\
0   &    \text{ otherwise}
\end{cases}
      . }   
      Now apply Remark~\ref{r:hilbert} with $H=L^2(m_\beta)$, $v=\Psi_{\beta,z,\undn}$,
      $v'=\Psi_{\beta,z',\undn'}$,    $u=\tilde \Psi_{\beta,z,\undn}$, $u'=\tilde \Psi_{\beta,z',\undn'}$ and $T\equiv T_\varphi$ the pointwise multiplication by $\varphi$
      to get
      \bel{e:otto}{ &\left| \dangle{\Psi_{\beta,z,\undn}}{\varphi \Psi_{\beta,z',\undn'} }_{L^2(m_\beta)}  -\dangle{\tilde \Psi_{\beta,z,\undn}}{\varphi \tilde \Psi_{\beta,z',\undn'}}_{L^2(m_\beta)}  \right|\le \\
     &  2 \,\|\varphi\|_{\infty}       \left(  \sqrt{1-\dangle{\tilde \Psi_{\beta,z,\undn}}{\Psi_{\beta,z,\undn}}^2_{L^2(m_\beta)}}      +  \sqrt{1-\dangle{\tilde \Psi_{\beta,z',\undn'}}{\Psi_{\beta,z',\undn'}}^2_{L^2(m_\beta)}}     \right)            .     }
Thus we obtain \eqref{e:measconv2} from Lemma~\ref{l:pavese}-(ii) and \eqref{e:gs}.

Thus the $(\Psi_{\beta,z,\undn})_{(z,\undn)\in S_{\lambda}}$ are an orthonormal base of $\mathrm{Range} (P_{\beta,\eps_{\Lambda},\lambda})$ satisfying \eqref{e:measconv2}. Since this holds for any $\lambda$ such that $S_\lambda$ is not empty and $L^2(m_\beta)$ is a direct sum of eigenspaces of $L_\beta$, the $(\Psi_{\beta,z,\undn})_{(z,\undn)}$ are an orthonormal base of $L^2(m_\beta)$ as $\lambda$ runs over the positive real numbers such that $S_\lambda$ is nonempty.
\end{proof}

\subsection{Eigenfunctions of the split spectrum}
\label{ss:HKN}
We denote by $\ell_{\beta,0}, \dots, \ell_{\beta,n}$ the first $n+1$ eigenvalues of $-\tfrac 1\beta L_\beta$ counted with multiplicity. It follows from
the smoothness of $V$,~\ref{a:2},~\ref{a:3}  and general principles that  $\ell_{\beta, 0}$ equals zero for every $\beta$ and is simple  
with constant eigenfunction $\Phi_{\beta,0}\equiv1/\sqrt{Z_\beta}$.
Moreover, with the notation of the previous Subsection~\ref{ss:simon}, $\lambda_{x,\underline{0}} = 0$ whenever $x\in\wp_0$ is a local minimum,
and $\lambda_{z, \undn}>0$ otherwise. It follows then from Theorem~\ref{t:simon} that
there exist $\eps>0$ and $\beta_0>0$ such that for each $\beta \ge \beta_0$
\bel{e:spec}{
\mathrm{Spec}(-\tfrac{1}{\beta} L_\beta)\cap [0,\eps]= \left\{\ell_{\beta,0}= 0,\ell_{\beta,1},\ldots,\ell_{\beta,n} \right\} .   
}
%It follows from a simple choice of orthonormal quasimodes, each of them being localized at a local minimum, and the min-max principle,
It is not difficult to see  that the eigenvalues  
$\ell_{\beta,1}, \dots, \ell_{\beta,n}$ are indeed exponentially small in $\beta$. A much stronger statement, 
giving the precise leading behaviour of the $\ell_{\beta,k}$'s in a general setting encompassing hypotheses~\ref{a:1}-\ref{a:5}, was given in 
\cite{HKN}, (see also~\cite{BG} and~\cite{E}). 
The following theorem conveniently resumes in our framework a weak version of the main result in~\cite{HKN}.
 \begin{theorem}[Low-lying eigenvalues]
 \label{t:HKN}
%There exists $\eps>0$ and $\beta_0>0$ such that for each $\beta \ge \beta_0$
%\bel{e:spec}{
%\mathrm{Spec}(-\tfrac{1}{\beta} L_\beta)\cap [0,\eps]= \left\{\ell_{\beta,0}= 0,\ell_{\beta,1},\ldots,\ell_{\beta,n} \right\} ,   
%}
%The eigenvalue $\ell_{\beta,0}=0$ is simple with constant eigenfunction $\Phi_{\beta,0}=1/\sqrt{Z_\beta}$.
For $k=1,\ldots,\,n$ the eigenvalue $\ell_{\beta,k}$ is simple, admits a normalized eigenfunction $\Phi_{\beta,k}\in C^\infty(M)$ and satisfies, for $\eta_k$ as in \eqref{e:coeffk},
\bel{e:lambdai}{
 \lim_{\beta}e^{ \beta W_k} 
\ell_{\beta,k}   =  \frac 12 \eta_k(x_k).
}
\end{theorem}

\begin{lemma}\label{l:montale}
For $k=1, \dots, n$ there exists $\tilde\Phi_{\beta, k}\in C^\infty_c( M)$ with $\|\tilde\Phi_{\beta, k} \|_{L^2(m_\beta)}=1$, such that
\begin{itemize}
\item[(i)]  the probability measure $\tilde \Phi^2_{\beta,k}  dm_\beta$ converges to $\delta_{x_k}$ in $\mc P(M)$.
\item[(ii)]  $\lim_{\beta \to \infty}\dangle{\tilde\Phi_{\beta, k}}{\Phi_{\beta, k}}^2 =  1   $ . 
\end{itemize}
\end{lemma}
\begin{proof}
In this proof the $W_k$'s are defined as in \ref{a:4}-\ref{a:5}, and we set for convenience $W_{n+1}=0$
and $\ell_{\beta,n+1}=\eps$, with $\eps$ as in \eqref{e:spec} (note that this is \emph{not} an eigenvalue of $-\tfrac{1}{\beta}L_\beta$).

Fix $\delta>0$ such that
\bel{e:delta}{
\delta < \min_{k=1,\ldots,n} (W_{k} - W_{k+1}).
}
Let $B_k\subset M$, respectively  $B_{k, \delta}$, be the connected component of $V^{-1}((-\infty, V(\hat x_k)))$, respectively $V^{-1}((-\infty, V(\hat x_k) -\delta))$, containing $x_k$. The $B_k$'s are precompact by Remark~\ref{r:helff}, and 
the closure of $B_{k, \delta}$ is contained in $B_k$. In particular there exists $\chi_k\in C_c^\infty(M;[0,1])$ such that $\chi_k \equiv 1$ on $B_{k,\delta}$ and $\chi_k \equiv 0$ on $M \setminus B_{k}$.

For $k=1,\ldots,n$ define $\tilde \Phi_{\beta,k}=\chi_k/\sqrt{Z_{\beta,k}}$ where
\bel{e:zbetak}{
Z_{\beta,k}=\int e^{-\beta V(x)}\chi_{k}^2(x)\,dx .
}
$V(x_k)$ is the unique minimum of $V(x)$ on $B_k$, and $\chi(x_k)=1$.
Therefore (i) holds. By the Laplace principle
\bel{e:laplace}{
 \lim_\beta \frac{1}{\beta} \log Z_{\beta,k} = -V(x_k) ,
}
\bel{e:laplace3}{
\lim_\beta\dangle{ \tilde \Phi_{\beta,j}}{\tilde \Phi_{\beta,k}}=0\qquad j,k=1,\ldots,n,\quad j\neq k,
}
\bel{e:laplace2}{
\lim_\beta\dangle{ \Phi_{\beta,0}}{\tilde \Phi_{\beta,k}}=0\qquad k=1,\ldots,n.
}
Note moreover that
\bel{e:stimagrad}{
\inf_{x\,:\: |\nabla \chi_k(x)|>0} V(x) \ge V(\hat x_k)-\delta= V(x_k)+W_k-\delta .
}
Denote now by $P_{\beta,k}:= \un {[\ell_{\beta,k+1},+\infty)}(-\tfrac{1}{\beta}L_\beta)$ the spectral projection of $-\tfrac{1}{\beta}L_\beta$ associated to the interval $[\ell_{\beta,k+1},+\infty)$. By spectral decomposition
\bel{e:proj}{
P_{\beta,k} \Phi=\Phi-\sum_{j=0}^k \dangle{\Phi}{\Phi_{\beta,j}} \Phi_{\beta,j}, \qquad     \forall \Phi\in L^2(m_\beta).
}
The Markov inequality for spectral projections, an integration by parts and \eqref{e:stimagrad} yield
\bel{e:markov1}{
\left\|P_{\beta,k} \tilde \Phi_{\beta,k} \right\|_{L^2(m_\beta)}^2 & \le  \frac{1}{\ell_{\beta,k+1}}
   		\dangle{\tilde \Phi_{\beta,k} }{-L_{\beta} \tilde \Phi_{\beta,k}}
= \frac{1}{\ell_{\beta,k+1}Z_{\beta,k}} \int  \!\! |\nabla \chi_k|^2\! (x) e^{-\beta V(x)} \!dx
\\ 
& \le C_{k} \frac{\exp(-\beta (V(x_k)+W_k-\delta))}{\ell_{\beta,k+1}\, Z_{\beta,k}} ,
}
where $C_k=\int |\nabla \chi_k|^2(x)dx<\infty$. The estimates of $\ell_{\beta,k+1}$ and $Z_{\beta,k}$ provided respectively in \eqref{e:lambdai} and \eqref{e:laplace}, and the choice \eqref{e:delta} of $\delta$ thus imply 
\bel{e:markov2}{
\lim_\beta \left\|P_{\beta,k} \tilde \Phi_{\beta,k} \right\|_{L^2(m_\beta)}=0  .
}
Taking the scalar product with $\tilde\Phi_{\beta,k}$ in both sides of \eqref{e:proj} calculated for $\Phi=\tilde\Phi_{\beta,k}$, gives
\bel{e:markov4}{
1- \dangle{\tilde \Phi_{\beta,k}}{\Phi_{\beta,k}}^2 
 & =\dangle{ \tilde \Phi_{\beta,k} }{P_{\beta,k} \tilde \Phi_{\beta,k}}
+ \sum_{j=0}^{k-1} \dangle{\tilde \Phi_{\beta,k}}{\Phi_{\beta,j}}^2
 \\ & 
=\|P_{\beta,k} \tilde \Phi_{\beta,k} \|_{L^2(m_\beta)}^2
+\dangle{\tilde \Phi_{\beta,k}}{\Phi_{\beta,0}}^2
\\
& \phantom{=}
 +\sum_{j=1}^{k-1} \Big(\dangle{\tilde \Phi_{\beta,j}}{\Phi_{\beta,j}} \dangle{\tilde \Phi_{\beta,k}}{\tilde \Phi_{\beta,j}} 
\\
&\phantom{=+\sum_{j=1}^{k-1} \Big(}
+ \dangle{\tilde \Phi_{\beta,k}}{\Phi_{\beta,j}-   \dangle{\tilde \Phi_{\beta,j}}{\Phi_{\beta,j}}\tilde \Phi_{\beta,j}} \Big)^2
\\
& \le \|P_{\beta,k} \tilde \Phi_{\beta,k} \|_{L^2(m_\beta)}^2
+\dangle{\tilde \Phi_{\beta,k}}{\Phi_{\beta,0}}^2
\\
& \phantom{=} +2\sum_{j=1}^{k-1}  \dangle{\tilde \Phi_{\beta,k}}{\tilde \Phi_{\beta,j}}^2
+1-\dangle{\tilde \Phi_{\beta,j}}{\Phi_{\beta,j}}^2 
\\
&=o_\beta(1)+2 \sum_{j=1}^{k-1} \left(1-\dangle{\tilde \Phi_{\beta,j}}{\Phi_{\beta,j}}^2\right),
}
where in the last line we used \eqref{e:laplace3}, \eqref{e:laplace2} and \eqref{e:markov2}.

In order to prove (ii), we now proceed by finite induction over $k$. The statement (ii) holds for $k=1$ as the sum in the last line of \eqref{e:markov4} is actually empty. But then from \eqref{e:markov4} one immediately yields the inductive step.

%%%, thanks to the explicit choice of $\tilde \Phi_{\beta,0}$.
%%%Suppose now that it is true for all $j<k$. Then
%%%\bel{e:markov3}{
%%%\tilde\Phi_{\beta,k}- \dangle{\tilde \Phi_{\beta,k}}{\Phi_{\beta,k}} \Phi_{\beta,k}=\sum_{j=0}^{k-1} \dangle{\tilde \Phi_{\beta,k}}{\Phi_{\beta,j}} \Phi_{\beta,j}+P_{\beta,k} \tilde \Phi_{\beta,k} 
%%%}
%%%and calculating the scalar product of both sides of this equality with $\tilde\Phi_{\beta,k}$
%%%\bel{e:markov4}{
%%%1- \dangle{\tilde \Phi_{\beta,k}}{\Phi_{\beta,k}}^2 
%%%& =\sum_{j=0}^{k-1} \dangle{\tilde \Phi_{\beta,k}}{\Phi_{\beta,j}}^2+\dangle{P_{\beta,k} \tilde \Phi_{\beta,k} }{P_{\beta,k} \tilde \Phi_{\beta,k}}
%%%\\
%%%& \le \sum_{j=0}^{k-1} (1-\dangle{\tilde \Phi_{\beta,j}}{\Phi_{\beta,j}})^2+\dangle{P_{\beta,k} \tilde \Phi_{\beta,k} }{P_{\beta,k} \tilde \Phi_{\beta,k}}
%%%}
%%%where in the last inequality we used the orthogonality of the $(\tilde \Phi_{\beta,j})_{j=1}^n$, which is immediate given that they have disjoint support.
%%%
%%%or otherwise stated
%%%\bel{e:markov5}{
%%%\lim_\beta \left\| \tilde\Phi_{\beta,k}-\sum_{j=0}^k \dangle{\tilde \Phi_{\beta,k}}{\Phi_{\beta,j}} \Phi_{\beta,j}\right\|_{L_2(m_\beta)}^2=0  .
%%%}
\end{proof}

\begin{proposition}\label{p:quasimodo}
For $k=0, \dots,n $ let $\Phi_{\beta,k}$ be a normalized eigenfunction corresponding to $\ell_{\beta,k}$ and define the probability measure $m_{\beta,k}$ by
\bel{e:mbetak}{
d m_{\beta,k}= \Phi^2_{\beta,k}\,  dm_\beta    .
}
Then $m_{\beta,k}$ converges to $\delta_{x_k}$ in $\mc P(M)$ as $\beta \to \infty$. In particular for $h\neq k$ it holds
\bel{e:mbetahk}{
\lim_{\beta} \int  |\Phi_{\beta,h}\,\Phi_{\beta,k}|  dm_\beta=0   .
}
\end{proposition}
\begin{proof}
The statement~\eqref{e:mbetak} is immediate for $k=0$, since we assumed $x_0$ to be the unique global minimizer of $V$. Let $k\ge 1$, take $\tilde \Phi_{\beta,k}$ as in Lemma~\ref{l:montale} and fix a test function $\varphi \in C_b(M)$. Now apply Remark~\ref{r:hilbert} with $H=L^2(m_\beta)$, $v=v'=\Phi_{\beta,k}$, $u=u'=\tilde \Phi_{\beta,k}$, and $T\equiv T_\varphi$ the pointwise multiplication by $\varphi$. We get from \eqref{e:tt}
\bel{e:conv1}{
\left| \int \varphi\, dm_{\beta,k}- \int  \varphi\,\tilde \Phi_{\beta,k}^2\, dm_\beta \right|
& =\left| \dangle{\Phi_{\beta,k}}{T_\varphi \Phi_{\beta,k}}-\dangle{\tilde \Phi_{\beta,k}}{T_{\varphi} \tilde \Phi_{\beta,k}}\right|
\\ & 
\le 4\,\|\varphi\|_{C(M)}
	\sqrt{1-\dangle{\tilde \Phi_{\beta,k}}{\Phi_{\beta,k}}^2}.
}
This vanishes by Lemma~\ref{l:montale}-(ii). Therefore by Lemma~\ref{l:montale}-(i)
\bel{e:conv2}{
\lim_\beta \int \varphi \,dm_{\beta,k}=\lim_\beta \int \varphi\,\tilde \Phi_{\beta,k}^2 \, dm_\beta(x)=
\varphi(x_k),
}
namely $m_{\beta,k}\to \delta_{x_k}$ as $\beta\to \infty$. Now \eqref{e:mbetahk} is an immediate consequence of this fact, since for $E\subset M$ a neighborhood of $x_h$ not containing $x_k$
\bel{e:hkpr}{
\int  |\Phi_{\beta,h}\,\Phi_{\beta,k}|  dm_\beta
\le \left(\int_{E}|\Phi_{\beta,k}|^2 dm_\beta\right)^{1/2}+\left(\int_{E^c}|\Phi_{\beta,h}|^2 dm_\beta\right)^{1/2},
}
and both terms vanish as $\beta\to\infty$.
\end{proof}

\section{$\Gamma$-convergence of $I_\beta$}
\label{s:gammauno}
\begin{proposition}
\label{p:dvcoer}
Assume $V\in C^2(M)$ and \ref{a:2}. Then $(I_\beta)$ is equicoercive.
\end{proposition}
\begin{proof}
Let $R(x):=\inf_{\beta \ge \beta_0} |\nabla V|^2(x)-\tfrac{2}{\beta} \Delta V(x)$. Taking $\psi=\beta\,\nabla V$ in the first line of \eqref{e:ex2} and
recalling \eqref{e:ibeta2}, one obtains for $\beta \ge  \beta_0$ and each $\mu\in \mc P(M)$,
\bel{e:coerc1}{
I_\beta(\mu)
& \ge \beta^{-2} \int d\mu \left ( \beta^2  |\nabla V|^2- \beta \Delta V-\tfrac{\beta^2}2  |\nabla V|^2\right)
\\ & =
 \tfrac 12 \int d\mu \left(|\nabla V|^2- \tfrac{2}{\beta} \Delta V\right)
\ge \tfrac 12 \int d\mu\,R . 
}
In particular, for any $t>0$ and $\mu \in \mc P(M)$,
it holds $\mu (\{R \le t\})\ge 1-2t^{-1} I_\beta(\mu)$. By \ref{a:2}, $R$ has compact sublevel sets, therefore any family $(\mu_\beta)\subset \mc P(M)$
such that $(I_\beta(\mu_\beta))$ is uniformly bounded, is tight. And, by Prohorov's Theorem, precompact.
\end{proof}

\begin{proof}[Proof of \eqref{e:gc1}, $\Gamma$-liminf]
Assume that $\mu_\beta \to \mu$. Then, replacing $\psi$ by $\beta\,\psi$ in \eqref{e:ex2}, one obtains for each $\psi\in \Omega$
\bel{e:stima11}{
I_\beta(\mu_\beta) \ge  \int d\mu_\beta \left(\nabla V \cdot \psi -\tfrac 1\beta \nabla \cdot \psi - \tfrac 12   \,\psi^2\right),
}
which implies
\bel{e:stima12}{
\varliminf_\beta I_\beta(\mu_\beta) \ge  \int d\mu \left(\nabla V \cdot \psi - \tfrac 12   \,\psi^2\right).
}
Optimizing over $\psi$ we get the result.
\end{proof}

\begin{proof}[Proof of \eqref{e:gc1}, $\Gamma$-limsup]
We proceed in three steps.

\textit{Step1}. First assume $\mu=\delta_{\bar x}$ for some $\bar x\in M$. Take $U\equiv U_{\bar x} \in C^2(M;[0,+\infty))$ a smooth function such that
\begin{itemize}
\item[(i)] $U(x)=\mathrm{distance}(x,\bar x)^2$ in a neighborhood of $\bar x$, and $U$ is strictly positive elsewhere.
\item[(ii)] If $M=\bb R^d$, then $U(x)\ge |x|^2+|\nabla V|^2(x)+|\Delta V|(x)$ for $x$ large enough and $\nabla\exp(-U)\in L^2(dx)$ (it is immediate to check that such a $U$ exists).
\end{itemize}
Let then
\bel{e:mubetadelta}{
\mu_{\beta,\bar x}(dx)=\frac{1}{C_{\beta}} \exp(-2\beta\,U(x))dx.
}
By the Laplace principle $\mu_{\beta,\bar x} \to \delta_{\bar x}$. On the other hand an explicit calculation shows
\bel{e:glimsupc}{
I_\beta(\mu_{\beta,\bar x})=\int d\mu_{\beta,\bar x} \,\left(\frac{|\nabla V|^2}2+2  |\nabla U|^2-\frac{1}{\beta}\Delta V\right).
}
If $M$ is compact, then
\bel{e:resto1}{
\lim_\beta \int d\mu_{\beta,\bar x} \,|\nabla U|^2 =|\nabla U|^2(\bar x)=0,
}
\bel{e:resto2}{
\lim_\beta \int d\mu_{\beta,\bar x}\, \Delta V =\Delta V(\bar x),
}
so that the second and third term in the left hand side of \eqref{e:glimsupc} vanish to get
\bel{e:ibetalimsup}{
\varlimsup_\beta I_\beta(\mu_{\beta,\bar x}) =\varlimsup_\beta \int d\mu_{\beta,\bar x} \,\frac{ |\nabla V|^2}2 = \frac{ |\nabla V|^2(\bar x)}2=I(\delta_{\bar x}).
}
If $M=\bb R^d$, \eqref{e:resto1}, \eqref{e:resto2} and \eqref{e:ibetalimsup} still hold if we restrict the integral to a (large) ball $B$ containing $\bar x$. Since  $e^{-2\beta U}/C_\beta$ vanishes pointwise out of $B$, 
it is enough to show that the integrands in 
 \eqref{e:resto1}, \eqref{e:resto2} and \eqref{e:ibetalimsup} are uniformly integrable. This is easily established by condition (ii) above.

\textit{Step2}. If $\mu=\sum_{k=1}^N \alpha_k \delta_{\bar x_k}$, then define $\mu_{\beta}=\sum_{k=1}^N \alpha_k \mu_{\beta,\bar x_k}$, where $\mu_{\beta,\bar x}$ is defined as in \eqref{e:mubetadelta}. Clearly $\mu_\beta \to \mu$. On the other hand, since $I_\beta$ is convex and $I$ is linear, again $\varlimsup_\beta I_\beta(\mu_\beta)\le I(\mu)$.

\textit{Step3}. By a standard diagonal argument, it is now enough to show that finite convex combinations of Dirac masses are $I$-dense in $\mc P(M)$. This is trivial, since $I$ is linear.
\end{proof}

\section{First order development by $\Gamma$-convergence}
\label{s:gammadue}

In this section we prove \eqref{e:gc2}.
The following remark is a restatement of \eqref{e:ex2}.
\begin{remark}
\label{r:ipure}
If $\mu=\varrho m_\beta$ and $h\in C^2(M)$ is such that $h^2=\varrho$, then 
\bel{e:idir}{
I_\beta(\mu)=-\frac{2}{\beta^2} \int dm_\beta(x) h(x) (L_\beta h)(x).
}
In particular if  $d\mu=h^2 dm_\beta$ and $h$ is an $L^2(m_\beta)$-normalized eigenfunction of $-\tfrac{1}{\beta} L_\beta$ with eigenvalue $\lambda$, then
\bel{e:ipure}{
I_\beta(\mu)=\frac{2}{\beta} \lambda .
}
More generally, for $P_{\beta,\eps,\lambda}$ defined as in \eqref{e:proj2} and for $h\in \mathrm{Range}(P_{\beta,\eps,\lambda})$
\bel{e:quasipure}{
\frac{2}{\beta} (\lambda-\eps)\le  I_\beta(\mu)\le \frac{2}{\beta} (\lambda+\eps).
}
%%
%%In particular for  $m_{\beta,k}$ as in \eqref{e:mbetak}
%%\bel{e:ipure2}{
%%I_\beta(m_{\beta,k})=\frac{2}{\beta}\ell_{\beta,k} \qquad k=0,\ldots,n.
%%}
\end{remark}

\begin{proof}[Proof of \eqref{e:gc2}: $\Gamma$-liminf inequality]
Let $\mu_\beta$ be a sequence converging to $\mu$ in $\mc P(M)$, we need to show that
\bel{e:ginfs}{
\varliminf_\beta  \beta I_\beta(\mu_\beta)\ge J(\mu).
}
From the $\Gamma$-liminf \eqref{e:gc1} proved in Section~\ref{s:gammauno}, the liminf in \eqref{e:ginfs} equals $+\infty$ unless  $I(\mu)=0$, namely $\mu$ has the form
\bel{e:mualpha}{
\mu=\sum_{z\in \wp}\alpha_z \delta_z.
}
Up to passing to a suitable subsequence (still labeled $\beta$ hereafter), we can assume that the liminf is actually a limit, so that with no loss of generality $I_\beta(\mu_\beta)<+\infty$ for all $\beta$ large enough. Then by Proposition~\ref{p:dvmore} we can assume that $\varrho_\beta=d\mu_\beta/dm_\beta$ is  in $C^2(M)$ and bounded away from $0$. In particular $h_\beta:=\sqrt{\varrho_\beta}\in C^2(M)$. Fix $\eps>0$, $\Lambda$ large enough and such that
 \bel{e:grandlamb}{
\Lambda >\max_{z\in \wp} \lambda_{z,\underline{0}},
}
and let $\beta \ge \max_{(z,\undn)\,:\:\lambda_{z,\undn}\le \Lambda} \beta_{\eps,z,\undn}$.
By Proposition~\ref{p:vittorini} and using the notation therein introduced, recalling also \eqref{e:lambdacard}-\eqref{e:proj2}, we can write
\bel{e:hpsi}{
& h_\beta =\sum_{\lambda \le \Lambda} P_{\beta,\eps,\lambda}h_\beta+P_{\beta,\Lambda}h_\beta    , }
where
\begin{gather*}
  P_{\beta,\Lambda} :=\un{(\Lambda,+\infty)}(-\tfrac{1}{\beta}L_\beta)   \      \  \    ,   \   \   \  
 P_{\beta,\eps,\lambda}h_\beta =
 \sum_{(z,\undn)\in S_\lambda} 
\gamma_{\beta,z,\undn} \Psi_{\beta,z,\undn}    \    ,   \\
\gamma_{\beta,z,\undn}:= \dangle{h_\beta}{\Psi_{\beta,z,\undn}}\in [-1,1].
\end{gather*}
By \eqref{e:idir}, Proposition~\ref{p:vittorini} and the Markov spectral inequality
%Since $\lambda_{z,\undn} \ge \lambda_{z,\underline{0}}$,  by \eqref{e:idir}, \eqref{e:quasipure}, \eqref{e:eigen}
\bel{e:dirgc1}{
\beta I_\beta(\mu_\beta) & =
2\sum_{\lambda \le \Lambda}  \dangle{-\tfrac{1}{\beta} L_\beta P_{\beta,\eps,\lambda} h_\beta}{ P_{\beta,\eps,\lambda} h_\beta}_{m_\beta}
\\ &
\phantom{ = 2\sum_{\lambda \le \Lambda}}
+
2 \dangle{-\tfrac{1}{\beta} L_\beta P_{\beta,\Lambda} h_\beta}{ P_{\beta,\Lambda} h_\beta}_{m_\beta}
\\
& \ge
2\sum_{\lambda \le \Lambda} (\lambda-\eps) \|P_{\beta,\eps,\lambda} h_\beta\|_{L^2(m_\beta)}^2
+
2 \Lambda \|P_{\beta,\Lambda} h_\beta\|_{L^2(m_\beta)}^2
\\ & =
2\sum_{\lambda \le \Lambda} \sum_{(z,\undn)\in S_\lambda} \gamma_{\beta,z,\undn}^2 (\lambda_{z,\undn} -\eps)+ 2 \Lambda \|P_{\beta,\Lambda} h_\beta\|_{L^2(m_\beta)}^2
\\ & \ge
2  \sum_{z\in \wp} \alpha_{\beta,z,\Lambda}( \lambda_{z,\underline{0}}-\eps)
+2  \Lambda \|P_{\beta,\Lambda} h_\beta\|_{L^2(m_\beta)}^2,
}
where we used in the last line $\lambda_{z,\undn} \ge \lambda_{z,\underline{0}}$ and introduced
\bel{e:alphabetaz}{
 \alpha_{\beta,z,\Lambda}=\sum_{\undn\,:\: \lambda_{z,\undn}\le \Lambda} \gamma_{\beta,z,\undn}^2.
}
As we take $\beta\to \infty$, \eqref{e:dirgc1} holds for any $\eps>0$, in particular we gather 
\bel{e:dirgc1c}{
 \varliminf_\beta 
\|P_{\beta,\Lambda} h_\beta\|_{L^2(m_\beta)}^2\le \varliminf_\beta \tfrac{1}{2\Lambda} \beta I_\beta(\mu_\beta)
}
and, since 
\bel{e:alphabetazz}{
\sum_{z\in \wp}\alpha_z =1=\|h_\beta\|_{L^2(m_\beta)}^2= \|P_{\beta,\Lambda} h_\beta\|_{L^2(m_\beta)}^2+\sum_{z\in \wp} \alpha_{\beta,z,\Lambda},
}
still by \eqref{e:dirgc1} and  the fact that $\eps$ is arbitrary (once we take $\beta\to\infty$)
\bel{e:dirgc1d}{
\varliminf_\beta \beta I_\beta(\mu_\beta)
& \ge \varliminf_\beta
2  \sum_{z\in \wp} \alpha_{\beta,z,\Lambda} \lambda_{z,\underline{0}}
+2 \Lambda \left(\sum_{z\in \wp}\alpha_z- \alpha_{\beta,z,\Lambda}\right)
\\ &
= 2 \sum_{z\in \wp} \alpha_z \lambda_{z,\underline{0}}+ 2 \varliminf_\beta  
\sum_{z\in \wp}(\alpha_z- \alpha_{\beta,z,\Lambda})(\Lambda -\lambda_{z,\underline{0}})
\\ &
=J(\mu)+ 2 \varliminf_\beta  
\sum_{z\in \wp}(\alpha_z- \alpha_{\beta,z,\Lambda})(\Lambda -\lambda_{z,\underline{0}}) ,
}
where in the last line we used \eqref{e:deflamb}. Since the sum above is finite, in view of \eqref{e:grandlamb}, we conclude once we show that
\bel{e:alphalamz}{
\varlimsup_\beta \alpha_{\beta,z,\Lambda}\le \alpha_{z},\qquad \forall z\in \wp.
}
To this aim, fix $z\in \wp$ and $\chi \in C_b(M)$, $0\le \chi\le 1$, such that $\chi(z)=1$ and $\chi(z')=0$ for $z'\in \wp$, $z'\neq z$. Then
\bel{e:mutest}{
\alpha_z& = \mu(\chi) 
=\lim_\beta \mu_\beta(\chi)=
\lim_\beta m_\beta(h_\beta^2 \chi)
\\ &
= \lim_\beta \sum_{\lambda,\lambda'\le \Lambda} \sum_{(z,\undn)\in S_\lambda}
 \sum_{(z',\undn')\in S_{\lambda'}} \gamma_{\beta,z,\undn}\gamma_{\beta,z',\undn'}
 \dangle{\chi \Psi_{\beta,z,\undn}}{\Psi_{\beta,z',\undn'}}_{L^2(m_\beta)}
 \\ &\phantom{=}
 		+ \lim_\beta \dangle{\chi P_{\beta,\Lambda}h_\beta}{2 h_\beta- P_{\beta,\Lambda}h_\beta}_{L^2(m_\beta)}.
}
Using \eqref{e:measconv2} both for the on- and off-diagonal terms in the above sum and recalling \eqref{e:alphabetaz} we obtain
\bel{e:mutest2}{
\alpha_z\ge \varlimsup_\beta \alpha_{\beta,z,\Lambda} - 2\|\chi\|_{C_b} \| P_{\Lambda,\beta}h_\beta\|_{L^2(m_\beta)}
\ge  \varlimsup_\beta \alpha_{\beta,z,\Lambda} -  \varliminf_\beta\sqrt{\frac{2\beta I_\beta(\mu_\beta)}{\Lambda} },
}
where in the last inequality we used \eqref{e:dirgc1c}. Since we can assume $\varliminf_\beta \beta I_\beta(\mu_\beta)=C<+\infty$ with no loss of generality, and since $\alpha_{\beta,z,\Lambda}$ is increasing in $\Lambda$, from \eqref{e:mutest2} we gather
\bel{e:mutest3}{
\varlimsup_\beta \alpha_{\beta,z,\Lambda}\le \inf_{\Lambda'\ge \Lambda} \varlimsup_\beta \alpha_{\beta,z,\Lambda'}\le
\alpha_z+  \inf_{\Lambda'\ge \Lambda} \sqrt{\frac{2C}{\Lambda} }=\alpha_z , 
}
namely \eqref{e:alphalamz}.
\end{proof}

\begin{proof}[Proof of \eqref{e:gc2}: $\Gamma$-limsup inequality]
If $J(\mu)=+\infty$ there is nothing to prove. So we can assume $\mu=\sum_{z\in \wp} \alpha_z \delta_{z}$. Recall the functions $(\Psi_{\beta,z,\undn})_{z\in\wp,\undn\in \bb N^d}$ introduced in Proposition~\ref{p:vittorini} and define $\mu_{\beta,z}\in \mc P(M)$ and the recovery sequence $\mu_\beta\in \mc P(M)$ by
\bel{e:muz}{
d\mu_{\beta,z}:= \Psi_{\beta,z,\underline 0}^2\,dm_\beta , 
}
\bel{e:murec1}{
\mu_{\beta}:=\sum_{z\in \wp} \alpha_z m_{\beta,z}.
}
By \eqref{e:measconv2}, $\mu_{\beta,z}\to \delta_z$ in $\mc P(M)$, so that $\mu_\beta\to \mu$.

Using the notation of Proposition~\ref{p:vittorini}, for each $\eps>0$ and for each $\beta\ge \max_{z\in \wp}\beta_{\eps,z,\underline{0}} $, by \eqref{e:quasipure} and \eqref{e:deflamb}
\bel{e:ibetaqu}{
I_\beta(m_{\beta,z})\le \tfrac{2}{\beta} (\lambda_{z,\underline{0}}+\eps)= \tfrac{1}{\beta} (\zeta(z)+2\eps).
}
Now, since $I_\beta$ is convex by Remark~\ref{r:lsc}
\bel{e:ibetag}{
\beta\,I_\beta(\mu_\beta)\le \beta \sum_{z\in \wp} \alpha_z I(m_{\beta,z})\le \sum_{z\in \wp} \alpha_z \,(\zeta(z)+2\eps)=J(\mu)+2\eps.
}
As we take the limsup $\beta\to +\infty$, \eqref{e:ibetag} holds for any $\eps>0$, and thus $\varlimsup_\beta I_\beta(\mu_\beta)\le J(\mu)$.
\end{proof}

\section{Exponential development by $\Gamma$-convergence}
\label{s:gammatre}
In this section we prove \eqref{e:gc3}.
\begin{proof}[Proof of \eqref{e:gc3}: $\Gamma$-liminf inequality]
Fix some $k=1,\ldots,n$ and let  $\mu_\beta$ be a sequence converging to $\mu$. We need to show that
\bel{e:ginf2}{
\varliminf_\beta  \beta e^{\beta W_k} I_\beta(\mu_\beta)\ge J_k(\mu).
}
Up to passing to a suitable subsequence (still labeled $\beta$ hereafter), we can assume that the liminf is actually a limit, so that with no loss of generality $I_\beta(\mu_\beta)<+\infty$ for all $\beta$ large enough. Then by Proposition~\ref{p:dvmore} we can assume that $\varrho_\beta=d\mu_\beta/dm_\beta$ is  in $C^2(M)$ and bounded away from $0$. In particular $h_\beta:=\sqrt{\varrho_\beta}\in C^2(M)$. As in the proof of Lemma~\ref{l:montale} we denote by $P_{\beta,k}:= \un {[\ell_{\beta,k+1},+\infty)}(-\tfrac{1}{\beta}L_\beta)$ the spectral projection of $-\tfrac{1}{\beta}L_\beta$ associated to the interval $[\ell_{\beta,k+1},+\infty)$. Then, since $h_\beta\in L^2(m_\beta)$, we can write
\bel{e:marlim}{
& h_\beta =\sum_{j=0}^k \gamma_{\beta,j} \Phi_{\beta,j}+P_{\beta,k} h_\beta,
\\
&\gamma_{\beta,j}= \dangle{h_\beta}{\Phi_{\beta,j}}\in [-1,1].
}
By \eqref{e:idir} and the spectral Markov inequality
\bel{e:liminf1a}{
I_\beta(\mu_\beta)& 
=
		\frac{2}{\beta} \sum_{j=0}^k \gamma^2_{\beta,j} \ell_{\beta,j}
		-\frac{2}{\beta^2} \int dm_\beta(x)\, P_{\beta,k}h_\beta (x)\,(L_\beta P_{\beta,k}h_\beta)(x) 
\\ &
\ge 
\frac{2}{\beta} \sum_{j=0}^{k} \gamma^2_{\beta,j}  \ell_{\beta,j}
				+\frac{2\ell_{\beta,k+1}}{\beta} \|P_{\beta,k}h_\beta\|_{L^2(m_\beta)}^2,
}				
i.e.				
\bel{e:liminf1b}{		
\beta e^{\beta W_k} I_\beta(\mu_\beta) 
 \ge &
2  \gamma^2_{\beta,k} \,e^{\beta W_k} \ell_{\beta,k}
\\ &
+
2\sum_{j=0}^{k-1} \gamma^2_{\beta,j} \,e^{\beta W_k} \ell_{\beta,j}
+2e^{\beta W_k} \ell_{\beta,k+1} \|P_{\beta,k}h_\beta\|_{L^2(m_\beta)}^2  . 
}

From \eqref{e:lambdai} and assumption \ref{a:5}, the terms $\gamma^2_{\beta,j} \,e^{\beta W_k} \ell_{\beta,j}$  with an index $j\le k-1$ vanish as $\beta\to\infty$, since $\gamma_{\beta,j}^2\le 1$. On the other hand, still by \eqref{e:lambdai} and \ref{a:5}, the term $e^{\beta W_k} \ell_{\beta,k+1} \|P_{\beta,k}h_\beta\|_{L^2(m_\beta)}^2$  diverges to $+\infty$ unless $\|P_{\beta,k}h_\beta\|_{L^2(m_\beta)} $ vanishes. Thus setting $\underline{ \alpha}:= \varliminf_\beta \gamma_{\beta,k}^2$ we gather again by \eqref{e:lambdai}
\bel{e:ilminf2}{
\varliminf_\beta \beta e^{\beta W_k} I_\beta(\mu_\beta) \ge 
\begin{cases}
+\infty & \text{if $\varliminf_\beta  \|P_{\beta,k}h_\beta\|_{L^2(m_\beta)}>0$,}
\\
\underline{ \alpha} \eta_k(x_k)
 & \text{otherwise.}
\end{cases}
}
In particular, if $\varliminf_\beta  \|P_{\beta,k}h_\beta\|_{L^2(m_\beta)}>0$, we proved the inequality \eqref{e:ginf2} no matter what the limit $\mu$ of the $\mu_\beta$ is.

Now we claim that if $\varliminf_\beta  \|P_{\beta,k}h_\beta\|_{L^2(m_\beta)}=0$, then necessarily $\lim_\beta \mu_\beta=\mu=\sum_{j=0}^k \alpha_j\delta_{x_j}$ for some $\alpha_j$, with $\alpha_k= \underline{ \alpha}$, so that $ \varliminf_\beta \beta e^{\beta W_k} I_\beta(\mu_\beta)\ge \alpha_k\,\eta_k(x_k)$, namely \eqref{e:ginf2}. To prove this claim note that
\bel{e:cuc}{
d\mu_\beta & =(h_\beta)^2 dm_\beta
\\ & =\left( \sum_{j=0}^k \gamma_{\beta,j}^2 \Phi_{\beta,j}^2\right)dm_\beta +
\left(\sum_{i\neq j,\,i,j=0}^k \gamma_{\beta,i}\gamma_{\beta,j}\Phi_{\beta,i}\Phi_{\beta,j}\right)dm_\beta
\\ & \phantom{=}
+(P_{\beta,k} h_\beta) \left( 2\sum_{j=0}^k \gamma_{\beta,j} \Phi_{\beta,j}+P_{\beta,k} h_\beta \right) dm_\beta.
}
Since $\|(P_{\beta,k} h_\beta) \|_{L^2(m_\beta)}$ vanishes  as $\beta\to \infty$ and \eqref{e:mbetahk} holds, the last two term above also vanish (weakly) as measures
as $\beta\to \infty$. On the other hand, since $\mu_\beta\to \mu$, it follows from Proposition~\ref{p:quasimodo} that $\alpha_j=\lim_\beta \gamma_{\beta,j}^2$ exists and $\mu=\sum_{j=0}^k \alpha_j\delta_{x_j}$.
\end{proof}

\begin{proof}[Proof of \eqref{e:gc3}: $\Gamma$-limsup inequality]
If $J_k(\mu)=+\infty$ there is nothing to prove. So we can assume $\mu=\sum_{j=0}^k \alpha_j\delta_{x_j}$. For $m_{\beta,j}$ as in \eqref{e:mbetak}, define the recovery sequence $\mu_\beta$ by 
\bel{e:mubetagc2}{
\mu_{\beta}:=\sum_{j=0}^k \alpha_j\, m_{\beta,j}.
}
By Proposition~\ref{p:quasimodo}, $\mu_\beta\to \mu$ in $\mc P(M)$. $I_\beta$ is convex, and by \eqref{e:ipure}
\bel{e:ibetagc2}{
I_\beta(\mu_\beta)\le \sum_{j=0}^k \alpha_j I(m_{\beta,j})=\frac{2}{\beta} \sum_{j=0}^k \alpha_j \ell_{\beta,j}
}
By \eqref{e:lambdai} and hypotheses \ref{a:5} one finally  gets
%only the $k$-th term in the sum does not vanish in the following limit, to yield
\bel{e:ibetaineq}{
\varlimsup_{\beta} \beta e^{\beta W_k} I_\beta(\mu_\beta)\le 2\,\alpha_k \lim_{\beta}  e^{\beta W_k}\ell_{\beta,k}=\alpha_k \eta_k(x_k)=J_k(\mu).
}
\end{proof}

\medskip

\noindent{\bf Acknowledgements:}   GDG gratefully acknowledges the financial support of the European Research Council under the European Union's Seventh
Framework Programme (FP/2007-2013) / ERC Grant Agreement number 614492, 
helpful discussions with Tony Leli\`{e}vre and Tom Hudson,  and the warm hospitality at the University of Rome La Sapienza during his frequent visits. 

This work has been carried out thanks to the support of the A*MIDEX
project (n.\ ANR11IDEX000102) funded by the \emph{Investissements
d'Avenir} French Government program, managed by the French National
Research Agency (ANR).

We are grateful to Lorenzo Bertini for introducing us to the problem.

\end{document}